\documentclass[a4paper,12pt]{amsart}
\usepackage{fullpage}
\newtheorem{thm}{Theorem}
\newtheorem{prop}{Proposition}
\newtheorem{lemma}{Lemma}

\DeclareSymbolFont{script}{U}{eus}{m}{n}
\DeclareMathSymbol{\Wedge}{0}{script}{"5E}

\newcommand{\intprod}{\makebox[10pt]{\rule{5pt}{.3pt}\rule{.3pt}{5pt}}}
\newcommand{\Rho}{\mathrm{P}}
\begin{document}
\title[Calabi Operator]
{A Calabi operator for\\ Riemannian locally symmetric spaces}
\author[E.F.~Costanza]{Federico Costanza}
\address{\hskip-\parindent
Centrum Fizyki Teoretycznej\\
Polska Akademia Nauk\\
02-668 Warszawa\\
Poland}
\email{efcostanza@gmail.com}
\author[M.G.~Eastwood]{Michael Eastwood}
\address{\hskip-\parindent
School of Mathematical Sciences\\
University of Adelaide\\ 
SA 5005\\ 
Australia}
\email{meastwoo@gmail.com}
\author[T.~Leistner]{Thomas Leistner}
\address{\hskip-\parindent
School of Mathematical Sciences\\
University of Adelaide\\ 
SA 5005\\ 
Australia}
\email{thomas.leistner@adelaide.edu.au}
\author[B.B.~McMillan]{Benjamin McMillan}
\address{\hskip-\parindent
Center for Complex Geometry\\
Institute for Basic Science\\
Daejeon 34126\\
South Korea}
\email{mcmillanbb@gmail.com}

\begin{abstract} On a Riemannian manifold of constant curvature, the Calabi
operator is a second order linear differential operator that provides local
integrability conditions for the range of the Killing operator.  We generalise
this operator to provide linear second order local integrability conditions on
Riemannian locally symmetric spaces, whenever this is possible.  Specifically,
we show that this generalised operator always works in the irreducible case and
we identify precisely those products for which it fails.
\end{abstract}

\renewcommand{\subjclassname}{\textup{2010} Mathematics Subject Classification}

\thanks{This work was supported by Australian Research Council (Discovery 
Project DP190102360).}

\maketitle

\setcounter{section}{-1}
\section{Introduction}\label{introduction}
On any Riemannian manifold, the {\em Killing operator\/} is the linear first
order differential operator ${\mathcal{K}}:\Wedge^1\to\bigodot^2\!\Wedge^1$ 
given by 
$$X_a\mapsto\nabla_{(a}X_{b)},$$
where $\nabla_a$ is the metric connection and, following Penrose's
{\em abstract index\/} notation~\cite{OT} as we shall do throughout this
article, round brackets on the indices mean to take the symmetric part.  With
its index raised using the metric, a vector field $X^a$ in the kernel of
${\mathcal{K}}$ is precisely a {\em Killing field\/}.  We define the 
{\em Calabi operator} to be a certain linear second order differential operator
${\mathcal{C}}$ acting on $\bigodot^2\!\Wedge^1$.  Specifically, it is given by
\begin{equation}\label{calabi}h_{ab}\mapsto
\nabla_{(a}\nabla_{c)}h_{bd}-\nabla_{(b}\nabla_{c)}h_{ad}
-\nabla_{(a}\nabla_{d)}h_{bc}
+\nabla_{(b}\nabla_{d)}h_{ac}-R_{ab}{}^e{}_{[c}h_{d]e}
-R_{cd}{}^e{}_{[a}h_{b]e},
\end{equation}
where $R_{ab}{}^c{}_d$ is the Riemann curvature tensor (and square brackets on
the indices mean to take the skew part). In~\cite{C}, Calabi showed that if the
Riemannian metric $g_{ab}$ is constant curvature, i.e.
\begin{equation}\label{cc}
R_{abcd}={\mathrm{constant}}\times(g_{ac}g_{bd}-g_{bc}g_{ad}),\end{equation}
then these two operators are the first two in a locally exact complex
$$\textstyle\Wedge^1\xrightarrow{\,{\mathcal{K}}\,}\bigodot^2\!\Wedge^1
\xrightarrow{\,{\mathcal{C}}\,}\cdots$$
of linear differential operators (the rest of which are first order).
Nowadays~\cite{E,EG}, this may be seen as an instance of the {\em
Bernstein-Gelfand-Gelfand\/} complex in flat projective differential geometry
(noting that, according to Beltrami's Theorem~\cite{B}, a Riemannian manifold
is projectively flat if and only if it is constant curvature). In
three-dimensional flat space, the operator ${\mathcal{C}}$ was introduced by
Saint-Venant in 1864 (as the integrability conditions for a {\em strain\/}
in continuum mechanics to arise from a {\em displacement\/}~\cite{SV}).

Our aim in this article is to construct, from the operator ${\mathcal{C}}$,
local integrability conditions for the range of the Killing operator.  We
choose to start with ${\mathcal{C}}$ because it works in the constant curvature
case \cite{C} and because, as observed by Gasqui and Goldschmidt~\cite{GG}, it
has the right symbol potentially to give second order constraints.  We remark
right away, however, that if we allow higher order constraints, then Gasqui and
Goldschmidt~\cite{GG} have already found a {\em third order} operator
\cite[Th\'eor\`eme~7.2]{GG} that provides local integrability conditions on the
range of~${\mathcal{K}}$ for all Riemannian locally symmetric spaces (our
methods easily reproduce \cite[Th\'eor\`eme~7.2]{GG}, as we briefly discuss
in~\S\ref{third_order}).  We should also remark that we have not been able to
find an elementary geometric origin for the operator~${\mathcal{C}}$, such as
deformation of Riemannian curvature (and a detailed discussion of this point
may be found in~\S\ref{deformation}).  In this article, ${\mathcal{C}}$ arises
na\"{\i}vely by {\em prolonging\/} the Killing equation, as discussed
in~\S\ref{generalities}.  A more sophisticated origin for ${\mathcal{C}}$ is
via {\em parabolic geometry\/}~\cite{parabook} and, in particular, the 
{\em second BGG operators\/} of Hammerl, Somberg, Sou\v{c}ek, and
\v{S}ilhan~\cite[\S4]{HSSS} (see~\S\ref{projective} for a brief discussion).
The range of ${\mathcal{K}}$ may be interpreted as a {\em gauge freedom\/} in
general relativity (as in \cite[(5.7.11)]{OT} and~\cite[(C.2.17)]{W}).  This
provides one motivation for our study.  On a compact Riemannian manifold, this
interpretation of the range of ${\mathcal{K}}$ is discussed
in~\cite[\S12.21]{Besse} (and is used there in conjunction with the {\em Ebin
Slice Theorem\/}).  The principal ingredient in our approach is the {\em
prolongation connection\/} of \S\ref{generalities}.  This was used already by
Kostant~\cite{Kostant} and Geroch~\cite[(B.2)]{Geroch} in similar contexts 
and is part of the general theory in~\cite{HSSS}.  Apart from algebraic
considerations, such as arising from the {\em Lie triple systems\/}~(\ref{LTS})
of locally symmetric spaces, our methods use only the general theory of
connections and their curvature.  As such, they are more widely applicable and
we pursue a general theory in a follow-up article~\cite{CELM2}.

Only for constant curvature is it the case that
${\mathcal{C}}\circ{\mathcal{K}}=0$. Nevertheless, in the locally symmetric
case, i.e.~$\nabla_aR_{bcde}=0$, this composition is relatively simple and one
can pick out just a part of ${\mathcal{C}}$, let us call it ${\mathcal{L}}$, so
that ${\mathcal{L}}\circ{\mathcal{K}}=0$. Under mild additional assumptions, we 
will show that this can be done
in such a way that ${\mathcal{L}}$ provides local integrability conditions for
the range of~${\mathcal{K}}$. More precisely, notice that the right hand side
of (\ref{calabi}) is a tensor 
satisfying Riemann tensor symmetries and, indeed, with Young tableau
notation~\cite{F}, the Calabi complex on a constant curvature Riemannian
manifold starts with
$$\Wedge^1=\begin{picture}(6,6)(0,0)
\put(0,0){\line(1,0){6}}
\put(0,6){\line(1,0){6}}
\put(0,0){\line(0,1){6}}
\put(6,0){\line(0,1){6}}
\end{picture}\xrightarrow{\,{\mathcal{K}}\,}
\begin{picture}(12,6)(0,0)
\put(0,0){\line(1,0){12}}
\put(0,6){\line(1,0){12}}
\put(0,0){\line(0,1){6}}
\put(6,0){\line(0,1){6}}
\put(12,0){\line(0,1){6}}
\end{picture}\xrightarrow{\,{\mathcal{C}}\,}
\begin{picture}(12,12)(0,2)
\put(0,0){\line(1,0){12}}
\put(0,6){\line(1,0){12}}
\put(0,12){\line(1,0){12}}
\put(0,0){\line(0,1){12}}
\put(6,0){\line(0,1){12}}
\put(12,0){\line(0,1){12}}
\end{picture}\xrightarrow{\,\phantom{\mathcal{C}}\,}
\begin{picture}(12,18)(0,6)
\put(0,0){\line(1,0){6}}
\put(0,6){\line(1,0){12}}
\put(0,12){\line(1,0){12}}
\put(0,18){\line(1,0){12}}
\put(0,0){\line(0,1){18}}
\put(6,0){\line(0,1){18}}
\put(12,6){\line(0,1){12}}
\end{picture}\xrightarrow{\,\phantom{\mathcal{C}}\,}\cdots.$$
The Riemann curvature tensor $R_{abcd}$ defines a homomorphism of vector 
bundles 
\begin{equation}\label{key_hom}
{\mathcal{R}}:\Wedge^2=\begin{picture}(6,12)(0,2)
\put(0,0){\line(1,0){6}}
\put(0,6){\line(1,0){6}}
\put(0,12){\line(1,0){6}}
\put(0,0){\line(0,1){12}}
\put(6,0){\line(0,1){12}}
\end{picture}
\longrightarrow\begin{picture}(12,12)(0,2)
\put(0,0){\line(1,0){12}}
\put(0,6){\line(1,0){12}}
\put(0,12){\line(1,0){12}}
\put(0,0){\line(0,1){12}}
\put(6,0){\line(0,1){12}}
\put(12,0){\line(0,1){12}}
\end{picture}\quad\mbox{given by}\quad\mu_{cd}\longmapsto
2R_{ab}{}^e{}_{[c}\mu_{d]e}+2R_{cd}{}^e{}_{[a}\mu_{b]e}\end{equation}
and, in case that $\nabla_aR_{bcde}=0$, the following diagram commutes.
\begin{equation}\label{diagram}\raisebox{-25pt}{\begin{picture}(100,58)
\put(0,45){\makebox(0,0){$\Wedge^1$}}
\put(-1,35){\vector(0,-1){20}}
\put(100,35){\vector(0,-1){20}}
\put(10,45){\vector(1,0){80}}
\put(10,5){\vector(1,0){80}}
\put(0,5){\makebox(0,0){$\Wedge^2$}}
\put(100,45){\makebox(0,0){$\begin{picture}(12,6)(0,0)
\put(0,0){\line(1,0){12}}
\put(0,6){\line(1,0){12}}
\put(0,0){\line(0,1){6}}
\put(6,0){\line(0,1){6}}
\put(12,0){\line(0,1){6}}
\end{picture}$}}
\put(100,5){\makebox(0,0){$\begin{picture}(12,12)(0,0)
\put(0,0){\line(1,0){12}}
\put(0,6){\line(1,0){12}}
\put(0,12){\line(1,0){12}}
\put(0,0){\line(0,1){12}}
\put(6,0){\line(0,1){12}}
\put(12,0){\line(0,1){12}}
\end{picture}$}}
\put(-7,26){\makebox(0,0){$d$}}
\put(50,52){\makebox(0,0){${\mathcal{K}}$}}
\put(50,12){\makebox(0,0){${\mathcal{R}}$}}
\put(93,26){\makebox(0,0){${\mathcal{C}}$}}
\end{picture}}\end{equation}
Therefore, if we form the quotient bundle
\begin{equation}\label{quotient_bundle}\overline{\begin{picture}(12,12)(0,2)
\put(0,0){\line(1,0){12}}
\put(0,6){\line(1,0){12}}
\put(0,12){\line(1,0){12}}
\put(0,0){\line(0,1){12}}
\put(6,0){\line(0,1){12}}
\put(12,0){\line(0,1){12}}
\end{picture}}\equiv\begin{picture}(12,12)(0,2)
\put(0,0){\line(1,0){12}}
\put(0,6){\line(1,0){12}}
\put(0,12){\line(1,0){12}}
\put(0,0){\line(0,1){12}}
\put(6,0){\line(0,1){12}}
\put(12,0){\line(0,1){12}}
\end{picture}\,/{\mathcal{R}}(\Wedge^2)\qquad
\big(\mbox{alternatively, take}\enskip
\overline{\begin{picture}(12,12)(0,2)
\put(0,0){\line(1,0){12}}
\put(0,6){\line(1,0){12}}
\put(0,12){\line(1,0){12}}
\put(0,0){\line(0,1){12}}
\put(6,0){\line(0,1){12}}
\put(12,0){\line(0,1){12}}
\end{picture}}\equiv{\mathcal{R}}(\Wedge^2)^\perp
\subseteq\begin{picture}(12,12)(0,2)
\put(0,0){\line(1,0){12}}
\put(0,6){\line(1,0){12}}
\put(0,12){\line(1,0){12}}
\put(0,0){\line(0,1){12}}
\put(6,0){\line(0,1){12}}
\put(12,0){\line(0,1){12}}
\end{picture}\,\big),\end{equation}
and write ${\mathcal{L}}$ for the composition
$$\begin{picture}(12,6)(0,0)
\put(0,0){\line(1,0){12}}
\put(0,6){\line(1,0){12}}
\put(0,0){\line(0,1){6}}
\put(6,0){\line(0,1){6}}
\put(12,0){\line(0,1){6}}
\end{picture}\xrightarrow{\,{\mathcal{C}}\,}
\begin{picture}(12,12)(0,2)
\put(0,0){\line(1,0){12}}
\put(0,6){\line(1,0){12}}
\put(0,12){\line(1,0){12}}
\put(0,0){\line(0,1){12}}
\put(6,0){\line(0,1){12}}
\put(12,0){\line(0,1){12}}
\end{picture}\xrightarrow{\,\phantom{\mathcal{C}}\,}
\overline{\begin{picture}(12,12)(0,2)
\put(0,0){\line(1,0){12}}
\put(0,6){\line(1,0){12}}
\put(0,12){\line(1,0){12}}
\put(0,0){\line(0,1){12}}
\put(6,0){\line(0,1){12}}
\put(12,0){\line(0,1){12}}
\end{picture}}\,,$$
then we obtain a complex of linear differential operators
\begin{equation}\label{key_complex}\begin{picture}(6,6)(0,0)
\put(0,0){\line(1,0){6}}
\put(0,6){\line(1,0){6}}
\put(0,0){\line(0,1){6}}
\put(6,0){\line(0,1){6}}
\end{picture}\xrightarrow{\,{\mathcal{K}}\,}
\begin{picture}(12,6)(0,0)
\put(0,0){\line(1,0){12}}
\put(0,6){\line(1,0){12}}
\put(0,0){\line(0,1){6}}
\put(6,0){\line(0,1){6}}
\put(12,0){\line(0,1){6}}
\end{picture}\xrightarrow{\,{\mathcal{L}}\,}
\overline{\begin{picture}(12,12)(0,2)
\put(0,0){\line(1,0){12}}
\put(0,6){\line(1,0){12}}
\put(0,12){\line(1,0){12}}
\put(0,0){\line(0,1){12}}
\put(6,0){\line(0,1){12}}
\put(12,0){\line(0,1){12}}
\end{picture}}\,.\end{equation}
The main result of this article is that, in most cases, this complex is
locally exact. The precise statement is as follows. 
\begin{thm}\label{main_thm}
Suppose $M$ is a Riemannian locally symmetric space. If we write $M$ as a 
product of irreducibles
$$M=M_1\times M_2\times\cdots\times M_k,$$
then the complex \eqref{key_complex} is locally exact unless $M$ has at least
one flat factor and at least one Hermitian factor, in which case
\eqref{key_complex} fails to be locally exact. (For example, the complex 
\eqref{key_complex} is
locally exact on $S^2\times S^2$ or $S^3\times S^1$ but not on 
$S^2\times S^1$.)
\end{thm}
In the statement of this theorem, we are supposing some preliminaries
concerning the theory of Riemannian locally symmetric spaces, specifically that
such spaces locally split as a product of {\em irreducibles\/} (those that
split no further).  The one-dimensional factors are flat.  Otherwise, there is
a rough divide into {\em Hermitian\/} and {\em non-Hermitian\/} types, where
Hermitian (as typified by complex projective space with its Fubini-Study
metric) is characterised by the existence of a non-zero $2$-form $\omega_{bc}$
with $\nabla_a\omega_{bc}=0$.  Moreover, these irreducible factors are Einstein
(with non-zero Einstein constant).  Both observations are due to the following
elementary fact.
\begin{lemma}\label{L-one}
Any parallel symmetric $2$-tensor $h_{ab}$ on an irreducible locally symmetric
space is a constant multiple of the metric.
\end{lemma}
\begin{proof}
If $h_{ab}$ is parallel and symmetric, the corresponding endomorphism 
$h^b{}_c\equiv g^{ab}h_{ac}$ is also parallel and is
diagonalisable over the reals with constant eigenvalues.  If there were two
distinct eigenvalues, the corresponding eigendistributions would be parallel,
and this already contradicts irreducibility.
\end{proof}
Since the Ricci tensor of a locally symmetric space is parallel, this lemma
implies that the metric is Einstein.  Moreover, it is well-known that the
Einstein constant has to be non-zero (for completeness, we prove this at the 
end of \S\ref{eigenvalues}).

If the locally symmetric space admits a non-zero parallel $2$-form $\omega_{ab}$, then
$\omega_a{}_c\omega^c{}_b$ is symmetric and parallel, hence a constant multiple
of the metric.  As $\omega_{ab}$ is skew, the corresponding endomorphism cannot
have real eigenvalues and therefore the constant multiple has to be negative,
in which case the metric is obliged to be K\"ahler with $\omega_{ab}$ a
constant multiple of the K\"ahler form~$J_{ab}$.  Helgason's classic text
\cite{H} provides further details on locally symmetric spaces and, indeed, a
complete classification (due to Cartan).  We shall not need this
classification.

The authors would like to thank Timothy Moy and Stuart Teisseire for their
invaluable contributions to a seminar series during which our proof of this
result was developed.  We would also thank Renato Bettiol for alerting us to
references~\cite{BG,G,KSW}.

Finally, we thank an anonymous referee for several valuable
suggestions/simplifications.

\section{Some generalities and the prolongation connection}
\label{generalities}
To begin, we need only assume that $g_{ab}$ is semi-Riemannian and locally
symmetric, meaning that $\nabla_aR_{bc}{}^d{}_e=0$, where $\nabla_a$ is the
Levi-Civita connection associated to $g_{ab}$ and $R_{ab}{}^c{}_d$ is the
curvature tensor of $\nabla_a$, characterised by
$$(\nabla_a\nabla_b-\nabla_b\nabla_a)X^c=R_{ab}{}^c{}_dX^d.$$
For a general semi-Riemannian metric, it is easy to check that
\begin{equation}\label{criterion}
h_{bc}=\nabla_{(b}\sigma_{c)}\quad\Leftrightarrow\quad
\left[\!\begin{array}{c}h_{bc}\\ 2\nabla_{[c}h_{d]b}
\end{array}\!\right]
=\left[\begin{array}{c}\nabla_b\sigma_c-\mu_{bc}\\ 
\nabla_b\mu_{cd}-R_{cd}{}^e{}_b\sigma_e
\end{array}\right],\enskip\mbox{for some }\mu_{bc}=\mu_{[bc]}\end{equation}
and we are, therefore, led to the {\em prolongation connection\/}:
\begin{equation}\label{prolongation_connection}
E\equiv\begin{array}{c}\Wedge^1\\[-4pt] \oplus\,\\[-1pt] \Wedge^2\end{array}
\ni\left[\!\begin{array}{c}\sigma_c\\ \mu_{cd}\end{array}\!\right]
\stackrel{\nabla_b}{\longmapsto}
\left[\begin{array}{c}\nabla_b\sigma_c-\mu_{bc}\\ 
\nabla_b\mu_{cd}-R_{cd}{}^e{}_b\sigma_e
\end{array}\right]\in\Wedge^1\otimes E\end{equation}
with curvature, in the locally symmetric case, given by
\begin{equation}\label{curvature}(\nabla_a\nabla_b-\nabla_b\nabla_a)
\left[\!\begin{array}{c}\sigma_c\\ \mu_{cd}\end{array}\!\right]
=\left[\!\begin{array}{c}0\\ 
2R_{ab}{}^e{}_{[c}\mu_{d]e}+2R_{cd}{}^e{}_{[a}\mu_{b]e}
\end{array}\!\right].\end{equation}
Notice that we are led to the homomorphism
${\mathcal{R}}:\Wedge^2\to\begin{picture}(12,12)(0,2) 
\put(0,0){\line(1,0){12}}
\put(0,6){\line(1,0){12}}
\put(0,12){\line(1,0){12}}
\put(0,0){\line(0,1){12}}
\put(6,0){\line(0,1){12}}
\put(12,0){\line(0,1){12}}
\end{picture}\subset\Wedge^2\otimes\Wedge^2$, appearing in the 
construction (\ref{key_hom}) of the operator~${\mathcal{L}}$
in~\S\ref{introduction}. We remark that the terminology `prolongation
connection' comes from `prolonging' the Killing equation: see
Proposition~\ref{prolongation_isomorphism}.

By differentiating the condition $\nabla_aR_{bc}{}^d{}_e=0$, we find that
\begin{equation}\label{LTS} R_{ab}{}^e{}_{[c}R_{d]e}{}^f{}_g
+R_{cd}{}^e{}_{[a}R_{b]e}{}^f{}_g=0.\end{equation} Indeed, a tensor with
Riemann tensor symmetries ($R_{abcd}=R_{[ab][cd]}$ and $R_{[abc]d}=0$) and
satisfying (\ref{LTS}) is called a {\em Lie triple system\/}~\cite{eschenbu,H}.
In any case, if we write $K\subseteq\Wedge^2$ for the kernel of
(\ref{key_hom}), then $R_{abcd}$ is actually a section of $K\odot
K\subseteq\Wedge^2\odot\Wedge^2$. \big(Although we shall not need the general
theory of locally symmetric spaces, we remark that the bundle $K$ is an
implicit feature of this theory~\cite{eschenbu}. Specifically, the Lie algebra
${\mathfrak{g}}$ of germs of Killing fields at a chosen basepoint~$p$, admits a
{\em Cartan decomposition\/}~\cite[\S4]{eschenbu}
$${\mathfrak{g}}={\mathfrak{k}}\oplus{\mathfrak{p}},$$
where ${\mathfrak{k}}$ denotes the subalgebra of Killing fields that vanish
at~$p$. As a vector space, we may identify ${\mathfrak{p}}$ with $T_pM$. The
Lie bracket $[{\mathfrak{k}},{\mathfrak{p}}]\subset{\mathfrak{p}}$ preserves
the metric $g_{ab}$ on $T_pM$ and we therefore obtain a Lie algebra
homomorphism ${\mathfrak{k}}\to\Wedge^2{\mathfrak{p}}^*$, which turns out to be
injective. The Cartan decomposition is sufficient to compute the Riemann
curvature tensor~\cite[Theorem~3]{eschenbu} and we find that the image of
${\mathfrak{k}}$ in $\Wedge^2{\mathfrak{p}}^*=\Wedge_p^2$ is~$K_p$.\big)
Observe that the subbundle $K\subseteq\Wedge^2$ is preserved by the Levi-Civita
connection. The formula (\ref{prolongation_connection}) now shows that the
subbundle
$$F\equiv\begin{array}{c}\;\Wedge^1\\[-4pt] \oplus\\[-1pt] K\end{array}
\subseteq
\begin{array}{c}\Wedge^1\\[-4pt] \oplus\;\\[-1pt] \Wedge^2\end{array}=E$$
is preserved by the prolongation connection and, by construction, this
subbundle is flat.

Restating the criterion (\ref{criterion}) in terms of the prolongation 
connection, we find that
$$h_{bc}=\nabla_{(b}\sigma_{c)}\quad\Leftrightarrow\quad
\left[\!\begin{array}{c}h_{bc}\\ 2\nabla_{[c}h_{d]b}
\end{array}\!\right]
=\nabla_b\left[\begin{array}{c}\sigma_c\\ 
\mu_{cd}\end{array}\right],\enskip\mbox{for some }\mu_{bc}=\mu_{[bc]}.$$
A straightforward calculation shows that, for $h_{bc}$ symmetric,
\begin{equation}\label{straightforward_calculation}
\nabla_a\left[\!\begin{array}{c}h_{bc}\\ 2\nabla_{[c}h_{d]b}
\end{array}\!\right]
-\nabla_b\left[\!\begin{array}{c}h_{ac}\\ 2\nabla_{[c}h_{d]a}
\end{array}\!\right]=\left[\!\begin{array}{c}0\\ ({\mathcal{C}}h)_{abcd}
\end{array}\!\right],\end{equation}
where ${\mathcal{C}}:\begin{picture}(12,6)(0,0)
\put(0,0){\line(1,0){12}}
\put(0,6){\line(1,0){12}}
\put(0,0){\line(0,1){6}}
\put(6,0){\line(0,1){6}}
\put(12,0){\line(0,1){6}}
\end{picture}\to\begin{picture}(12,12)(0,2)
\put(0,0){\line(1,0){12}}
\put(0,6){\line(1,0){12}}
\put(0,12){\line(1,0){12}}
\put(0,0){\line(0,1){12}}
\put(6,0){\line(0,1){12}}
\put(12,0){\line(0,1){12}}
\end{picture}\subset\Wedge^2\otimes\Wedge^2$
is the operator~(\ref{calabi}). It is interesting to note, however, that this 
straightforward calculation gives several alternative formul{\ae} 
for~${\mathcal{C}}$, for example
\begin{equation}\label{alternative_calabi}
\textstyle h_{ab}\mapsto 
2\big[\nabla_c\nabla_{[a}h_{b]d}-\nabla_d\nabla_{[a}h_{b]c}
-R_{ab}{}^e{}_{[c}h_{d]e}\big]\end{equation}
(and this one will be useful later).
In any case, from (\ref{curvature}), we
arrive at a {\em necessary\/} condition for $h_{ab}$ to be in the range of the
Killing operator, namely that $({\mathcal{C}}h)_{abcd}$ be in the range of the
homomorphism~(\ref{key_hom}). This is exactly the condition that was built into
the operator ${\mathcal{L}}$ in \S\ref{introduction} and reformulated there as
(\ref{key_complex}) being a complex. 

To summarise \S\ref{generalities}, in the prolongation connection
(\ref{prolongation_connection}) we have found a geometric \mbox{formulation} of
the necessary condition to be in the range of the Killing operator, afforded by
commutativity of the diagram~(\ref{diagram}). This reasoning pertains for any
{\em semi}-Riemannian locally symmetric metric but now we ask whether this
condition is locally {\em sufficient\/}, and to answer this question we shall
now restrict attention to the {\em Riemannian\/} case.

\section{The Riemannian locally symmetric case}
We set $C\equiv K^\perp\subseteq\Wedge^2$ so that we have an orthogonal
decomposition $\Wedge^2=K\oplus C$ preserved by~$\nabla_a$. In particular, if
$\mu_{cd}\in\Gamma(C)$, then $R_{ab}{}^e{}_{[c}\mu_{d]e}\in\Gamma(K\otimes C)$.
Now, by construction, the homomorphism 
$$\Wedge^2\supseteq C\ni\mu_{cd}\stackrel{{\mathcal{R}}}{\longmapsto}
2R_{ab}{}^e{}_{[c}\mu_{d]e}+2R_{cd}{}^e{}_{[a}\mu_{b]e}\in
\Wedge^2\otimes\Wedge^2$$
is injective but, since $2R_{ab}{}^e{}_{[c}\mu_{d]e}\in\Gamma(K\otimes C)$,
it follows that
\begin{equation}\label{LCcurv}
C\longrightarrow K\otimes C\subseteq\Wedge^2\otimes C\quad\mbox{given by}
\quad\mu_{cd}\longmapsto 2R_{ab}{}^e{}_{[c}\mu_{d]e}\end{equation}
is injective. We may feed these observations back into 
(\ref{criterion}) in an attempt to find the range of the Killing 
operator. Looking back at (\ref{prolongation_connection}), 
we may use the decomposition
$$\mbox{\large $E={}\quad$}
\raisebox{-10pt}{$\Wedge^2=\left\lbrace\rule{0pt}{20pt}\right.$}
\hspace{-7pt}\begin{array}{c}
\Wedge^1\\[-4pt] \oplus\\[-1pt] K\\[-4pt] \oplus\\[-1pt] C\end{array}
\hspace{-6pt}\raisebox{13pt}{$\left.\rule{0pt}{20pt}\right\rbrace=F$}$$
to rewrite the prolongation connection as
\begin{equation}\label{refined}E=\begin{array}{c}
\Wedge^1\\[-4pt] \oplus\\[-1pt] K\\[-4pt] \oplus\\[-1pt] C\end{array}
\ni\left[\!\begin{array}{c}\sigma_c\\ \lambda_{cd}\\ \theta_{cd}
\end{array}\!\right]
\stackrel{\nabla_b}{\longmapsto}
\left[\begin{array}{c}\nabla_b\sigma_c-\lambda_{bc}-\theta_{bc}\\ 
\nabla_b\lambda_{cd}-R_{cd}{}^e{}_b\sigma_e\\ \nabla_b\theta_{cd}
\end{array}\right]\in\Wedge^1\otimes E\end{equation}
with curvature given by
\begin{equation}\label{refined_curvature}(\nabla_a\nabla_b-\nabla_b\nabla_a)
\left[\!\begin{array}{c}
\sigma_c\\ \lambda_{cd}\\ \theta_{cd}\end{array}\!\right]
=\left[\!\begin{array}{c}0\\ 
2R_{cd}{}^e{}_{[a}\theta_{b]e}\\
2R_{ab}{}^e{}_{[c}\theta_{d]e}
\end{array}\!\right]
\in\begin{array}{c}\Wedge^2\otimes\Wedge^1\\[-4pt] 
\oplus\\[-3pt] C\otimes K\\[-4pt] \oplus\\[-3pt] K\otimes C
\end{array}\!\!.\end{equation}
If we write
$$\left[\!\begin{array}{c}h_{bc}\\ 2\nabla_{[c}h_{d]b}
\end{array}\!\right]\in\Wedge^1\otimes E\quad\mbox{as}\quad
\left[\!\begin{array}{c}\sigma_{bc}\\ \lambda_{bcd}\\ \theta_{bcd}
\end{array}\!\right]\in\;
\Wedge^1\otimes
\begin{array}{c}\Wedge^1\\[-4pt] \oplus\\[-3pt] K\\[-4pt] \oplus\\[-3pt]
C\end{array},$$
then, from (\ref{refined}) and~(\ref{refined_curvature}), we see that
$({\mathcal{C}}h)_{abcd}$ being in the range of the homomorphism
(\ref{key_hom}) implies that
\begin{equation}\label{key_implication}
\nabla_a\theta_{bcd}-\nabla_b\theta_{acd}=2R_{ab}{}^e{}_{[c}\theta_{d]e}
\end{equation}
for some $\theta_{cd}\in\Gamma(C)\subseteq\Gamma(\Wedge^2)$. We have already 
noted that (\ref{LCcurv}) 
is injective. Therefore, the section $\theta_{cd}\in\Gamma(C)$ is uniquely
determined and, if $\theta_{bcd}\in\Wedge^1\otimes C$ is to be in the range of
the Levi-Civita connection $C\to\Wedge^1\otimes C$ (as is necessary
from~(\ref{refined})), then the only possibility is that
$\theta_{bcd}=\nabla_b\theta_{cd}$. {\em Assuming this to be the case}, we may
now consider
$$\left[\!\begin{array}{c}\sigma_{bc}\\ \lambda_{bcd}\\ \theta_{bcd}
\end{array}\!\right]
-\nabla_b\left[\!\begin{array}{c}0\\ 0\\ \theta_{cd}
\end{array}\!\right]
=\left[\!\begin{array}{c}\sigma_{bc}+\theta_{bc}\\ \lambda_{bcd}\\ 0
\end{array}\!\right].$$
This is a section of $\Wedge^1\otimes F$ and, by construction, maps to zero 
under $\Wedge^1\otimes F\to\Wedge^2\otimes F$. Since $F$ is flat, the sequence 
$$F\xrightarrow{\,\nabla\,}\Wedge^1\otimes F
\xrightarrow{\,\nabla\,}\Wedge^2\otimes F$$
is a locally exact complex and we have proved the following. 
\begin{prop}\label{prop1}
Suppose $M$ is a Riemannian locally symmetric space. In order to show that the
complex \eqref{key_complex} is locally exact, it suffices to show that if
$\theta_{bcd}\in\Gamma(\Wedge^1\otimes C)$ satisfies \eqref{key_implication}
for some (uniquely determined)
$\theta_{cd}\in\Gamma(C)\subseteq\Gamma(\Wedge^2)$, then
$\nabla_b\theta_{cd}=\theta_{bcd}$.
\end{prop}
In an attempt to use this proposition, let us write out (\ref{key_implication})
more fully
$$\nabla_b\theta_{cde}-\nabla_c\theta_{bde}
=R_{bc}{}^f{}_d\theta_{ef}-R_{bc}{}^f{}_e\theta_{df}$$
and apply $\nabla_a$, firstly noting that
$$-\nabla_a\nabla_c\theta_{bde}
=-\nabla_c\nabla_a\theta_{bde}+R_{ac}{}^f{}_b\theta_{fde}
+R_{ac}{}^f{}_d\theta_{bfe}+R_{ac}{}^f{}_e\theta_{bdf},$$
to conclude that
$$\nabla_a\nabla_b\theta_{cde}-\nabla_c\nabla_a\theta_{bde}
+R_{ac}{}^f{}_b\theta_{fde}
=R_{bc}{}^f{}_d\nabla_a\theta_{ef}+R_{ac}{}^f{}_d\theta_{bef}
-R_{bc}{}^f{}_e\nabla_a\theta_{df}
-R_{ac}{}^f{}_e\theta_{bdf}.$$
Skewing this equation over $ab$ and using (\ref{key_implication}) to substitute
for $\nabla_c\nabla_{[a}\theta_{b]de}$ leads to
\begin{equation}\label{old_conclusion_by_new_route}
R_{ab}{}^f{}_dX_{cef}+R_{ca}{}^f{}_dX_{bef}+R_{bc}{}^f{}_dX_{aef}
-R_{ab}{}^f{}_eX_{cdf}-R_{ca}{}^f{}_eX_{bdf}-R_{bc}{}^f{}_eX_{adf}=0,
\end{equation}
where $X_{cef}\equiv\theta_{cef}-\nabla_c\theta_{ef}$. According to 
Proposition~\ref{prop1}, we have proved the 
following.
\begin{prop}\label{prop2}
Suppose $M$ is a Riemannian locally symmetric space. In order to show that the
complex \eqref{key_complex} is locally exact, it suffices to show that if
$X_{bcd}\in\Gamma(\Wedge^1\otimes C)$ satisfies
equation~\eqref{old_conclusion_by_new_route}, then $X_{bcd}=0$.
\end{prop}
We are now in a position to prove Theorem~\ref{main_thm} in case that $M$ is
irreducible.
\section{The irreducible case}
Although (\ref{old_conclusion_by_new_route}) is a
seemingly powerful equation in $\Wedge^3\otimes\Wedge^2$, it is difficult
to use in this form and, instead, we shall use only its trace, which yields an 
equation in $\Wedge^2\otimes\Wedge^1$, namely (bearing in mind that 
$R_{abcd}\in\Gamma(K\otimes K)$ and that $C$ is orthogonal to~$K$)
\begin{equation}\label{trace}
R_b{}^dX_{ade}-R_a{}^dX_{bde}=R_{ab}{}^{cd}X_{cde}-R_{ab}{}^d{}_eX^c{}_{cd},
\end{equation}
where $R_b{}^d\equiv R_{ab}{}^{ad}$ is the Ricci tensor. 
In the irreducible case, the metric is Einstein, 
i.e.~$R_{ab}=\lambda g_{ab}$ for some constant~$\lambda$. There are three cases 
according to the sign of~$\lambda$:
\begin{itemize}
\item $\lambda>0$ and $M$ is said to be of {\em compact\/} type,
\item $\lambda<0$ and $M$ is said to be of {\em non-compact\/} type,
\item $\lambda=0$, in which case $M$ is flat (and one-dimensional).
\end{itemize}
For flat Euclidean space ${\mathbb{R}}^p$ of any dimension or, indeed, any
constant curvature metric of any signature, the homomorphism (\ref{key_hom})
vanishes. Consequently, the subbundle $K\subseteq\Wedge^2$ coincides
with~$\Wedge^2$ and $C=K^\perp$ vanishes. In this case, the criterion of
Proposition~\ref{prop2} is vacuously satisfied and we have recovered Calabi's 
result that (\ref{key_complex}) is locally exact for constant curvature 
metrics. Hence, we may suppose $\lambda\not=0$ and, without loss of 
generality, rescale the metric so that $R_{ab}=g_{ab}$ for compact type or 
$R_{ab}=-g_{ab}$ for non-compact.

The following Lemma now suffices to establish Theorem~\ref{main_thm} for
irreducible~$M$.

\begin{lemma}\label{algebraic_result_for_LTS}
Suppose $M$ is irreducible and $X_{abc}$ is a section of $\Wedge^1\otimes C$.
If $X_{abc}$ satisfies \eqref{trace} and
$R_{ab}=\pm g_{ab}$, then $X_{abc}=0$.
\end{lemma}
\begin{proof}
Suppose that $M$ is of compact type, i.e.~$R_{ab}=g_{ab}$. 
If we set 
\begin{equation}\label{XvY}
Y_{abc}\equiv X_{[ab]c},\quad\mbox{equivalently}\enskip 
X_{abc}=Y_{abc}+Y_{cab}+Y_{cba},\end{equation}
then (\ref{trace}) now reads
$$2Y_{abe}=R_{ab}{}^{cd}X_{cde}-R_{ab}{}^d{}_eX^c{}_{cd}$$
and tracing over $be$ gives
$$2Y_{ab}{}^b=R_a{}^{bcd}X_{cdb}-X^c{}_{ca}.$$
On the other hand, tracing $Y_{abc}=X_{[ab]c}$ gives
$$2Y_{ab}{}^b=X^b{}_{ba}.$$
It follows that
$$4X^c{}_{ca}=2R_a{}^{bcd}X_{cdb}=-R_a{}^{cdb}X_{cdb},$$
where this last equality follows from the Bianchi symmetry $R_a{}^{[bcd]}=0$.
However, recall that (\ref{LTS}) implies $R_{abcd}\in\Gamma(K\otimes K)$ whilst
$X_{abc}\in\Gamma(\Wedge^1\otimes C)$. Therefore $R_a{}^{cdb}X_{cdb}=0$. It
follows that $X^c{}_{ca}=0$ and, since
$R_{ab}{}^{cd}X_{cde}=R_{ab}{}^{cd}Y_{cde}$, equation (\ref{trace}) now yields
\begin{equation}\label{Yeq}
R_{ab}{}^{cd}Y_{cde}=2Y_{abe}.\end{equation} 
We are, therefore, led to the operator
\begin{equation}\label{R}
R:\Wedge^2\to\Wedge^2\quad\mbox{given by}\quad\omega_{ab}\longmapsto
R_{ab}{}^{cd}\omega_{cd}\end{equation} and its eigenvalues. The interchange
symmetry for $R_{abcd}$ says that $R$ is symmetric and is, therefore,
orthogonally diagonalisable. In the following section we shall show that, for
locally symmetric irreducible compact type, normalised so that $R_{ab}=g_{ab}$,
we have
$$R_{ab}{}^{cd}\omega_{cd}=\lambda\omega_{ab}\enskip\mbox{for}\enskip
\omega_{ab}\not=0\quad\Rightarrow\quad 0\leq\lambda\leq 2$$
with $\lambda=2$ only in the Hermitian case. Moreover, the $2$-eigenspace in
the Hermitian case is spanned by $J_{ab}$, the K\"ahler form. In particular,
equation (\ref{Yeq}) implies that either $Y_{abc}=0$ in the non-Hermitian case,
or $Y_{abc}=J_{ab}\phi_c$ for some $1$-form~$\phi_c$ in the Hermitian case. But
now, in the Hermitian case,
$$0=Y_{ab}{}^b=J_{ab}\phi^b$$
and $J_{ab}$ being nondegenerate implies that $\phi_c=0$. Thus, in all cases,
we conclude that $Y_{abc}=0$ and, therefore, that $X_{abc}=0$, as required.

For non-compact type we may normalise the Ricci tensor
$R_{ab}=-g_{ab}$, equation (\ref{Yeq}) is replaced by
$$R_{ab}{}^{cd}Y_{cde}=-2Y_{abe},$$
eigenvalues are in the range $-2\leq\lambda\leq0$, and, similarly, the only
available eigenvector with eigenvalue $-2$ is the K\"ahler form in the 
Hermitian case. Again, we conclude that $X_{abc}=0$, as required.
\end{proof}

\section{Eigenvalues of the curvature operator}\label{eigenvalues}
In this section we shall establish the result about the eigenvalues of the
Riemann curvature as an operator on $2$-forms that was needed in the proof of
Lemma~\ref{algebraic_result_for_LTS}.  This result may also be gleaned from
\cite[Table (39)]{G} (cf.~\cite{BG,BK,KSW}).  However, since we use a different
normalisation for the metric and have a straightforward proof available, we
present it here.
\begin{thm}\label{eigenvalues_of_curvature}
Suppose that $M$ is an irreducible Riemannian locally symmetric 
space and consider the endomorphism~\eqref{R}. If $M$ is compact type, with 
Ricci tensor normalised so that $R_{ab}=g_{ab}$, then the eigenvalues of this 
endomorphism lie in the interval $[0,2]$. Regarding the end points of this 
interval,
\begin{itemize}
\item 
$R_{ab}{}^{cd}\omega_{cd}=0\quad\Leftrightarrow\quad\omega_{ab}\in\Gamma(C)$;
\item $R_{ab}{}^{cd}\omega_{cd}=2\omega_{ab}\enskip\mbox{for}\enskip
\omega_{ab}\not=0\quad\Leftrightarrow\quad M$ is Hermitian and
$\omega_{ab}$ is a smooth multiple of the K\"ahler form, 
i.e.~$\omega_{ab}=\lambda J_{ab}$ for some smooth function~$\lambda$.
\end{itemize}
If $M$ is non-compact type, with Ricci tensor normalised so that
$R_{ab}=-g_{ab}$, then the eigenvalues of \eqref{R} lie in the interval
$[-2,0]$ and
\begin{itemize}
\item 
$R_{ab}{}^{cd}\omega_{cd}=0\quad\Leftrightarrow\quad\omega_{ab}\in\Gamma(C)$;
\item $R_{ab}{}^{cd}\omega_{cd}=-2\omega_{ab}\enskip\mbox{for}\enskip
\omega_{ab}\not=0\enskip\Leftrightarrow\enskip M$ is Hermitian and
$\omega_{ab}$ is a smooth multiple of the K\"ahler form,
i.e.~$\omega_{ab}=\lambda J_{ab}$ for some smooth function~$\lambda$.
\end{itemize}
\end{thm}
\begin{proof} Recall from (\ref{LTS}) that $R_{abcd}$ is a section of
$K\bigodot K\subseteq\Wedge^2\bigodot\Wedge^2$.  Therefore, as an endomorphism
(\ref{R}) of~$\Wedge^2$, it is self-adjoint and annihilates~$C\equiv K^\perp$.
This endomorphism therefore preserves the orthogonal decomposition
$\Wedge^2=K\oplus C$.  Thus, we may restrict 
$\omega_{ab}\mapsto R_{ab}{}^{cd}\omega_{cd}$ to $K\subseteq\Lambda^2$ and,
supposing that $R_{ab}=g_{ab}$, we are required to show that its eigenvalues
(necessarily real) lie in the interval $(0,2]$ and to prove the stated
consequences of having an eigenvalue equal to~$2$. Continuing to suppose that 
$R_{ab}=g_{ab}$, tracing (\ref{LTS}) over the indices
$dg$ yields 
\begin{equation}\label{LTS_traced}\textstyle R_{ab}{}_{cd}=
\frac12R_{ab}{}^{ef}R_{cdef}+R_a{}^{ef}{}_cR_{befd}-R_b{}^{ef}{}_cR_{aefd}.
\end{equation}
Therefore, if $R_{ab}{}^{cd}\omega_{cd}=\lambda\omega_{ab}$, then 
$$\textstyle\lambda\omega_{ab}=\frac12\lambda^2\omega_{ab}
+2R_a{}^{ef}{}_cR_{befd}\omega^{cd},$$
whence
$$\textstyle(\lambda-\frac12\lambda^2)\omega^{ab}\omega_{ab}
=2\omega^{ab}R_a{}^{efc}R_{bef}{}^d\omega_{cd}
=2\omega^{a[b}R_a{}^{e]fc}R_{be[f}{}^d\omega_{c]d}.$$
Recall that we are restricting to the case where $\omega_{ab}$ lies in $K$,
which, by definition, implies that
$$R_{be[f}{}^d\omega_{c]d}=-R_{fc[b}{}^d\omega_{e]d}$$
and we conclude that
$$\textstyle(\lambda-\frac12\lambda^2)\|\omega_{ab}\|^2
=(\lambda-\frac12\lambda^2)\omega^{ab}\omega_{ab}
=-2\omega^{a[b}R_a{}^{e]fc}R_{fc[b}{}^d\omega_{e]d}
=2\|R_{fc[b}{}^d\omega_{e]d}\|^2.$$
It follows that $\lambda-\frac12\lambda^2\geq 0$, in other words that
$\lambda\in[0,2]$. Furthermore, if $\lambda=0$ or $\lambda=2$,
then $R_{ab}{}^e{}_{[c}\omega_{d]e}=0$. Tracing this equation over $bc$ gives
$$\textstyle\omega_{da}=R_a{}^e\omega_{de}=R_a{}^{be}{}_d\omega_{be}=
\frac12R_{da}{}^{be}\omega_{be},\quad\mbox{(forcing $\lambda=2$)},$$
this last equality by the Bianchi symmetry in the form
$R_{a[be]d}=\frac12R_{dabe}$.  Hence, the possibility that $\lambda=0$ for
$\omega_{ab}$ a section of $K$ is eliminated.  We are left with
$\lambda\in(0,2]$, with $\lambda=2$ implying that $\omega_{ab}$ is a section of
the bundle
\begin{equation}\label{two-forms_killed_by_curvature}
\{X_{ab}\in\Wedge^2\mid R_{ab}{}^e{}_{[c}X_{d]e}=0\},
\end{equation} 
which, by local symmetry, is parallel and manifestly flat.  In this case,
therefore, this bundle admits a non-zero parallel section, a $2$-form
$\tilde{\omega}_{ab}$.  As noted just after Lemma~\ref{L-one}, this forces $M$
to be Hermitian and $\tilde\omega_{ab}$ to be a constant multiple of~$J_{ab}$,
the K\"ahler form.  We conclude the bundle
(\ref{two-forms_killed_by_curvature}) is rank one, and that $\omega_{ab}$
itself is a smooth multiple of~$J_{ab}$.


For irreducible non-compact type normalised so that $R_{ab}=-g_{ab}$, the same
argument applies {\em mutatis mutandis\/}.  Details are left to the reader.
\end{proof}

{\bf Remark.}\enskip A simple variation on this argument proves the classical
result that, if a locally symmetric Riemannian metric is Ricci-flat, then it is
flat.  Specifically, tracing (\ref{LTS}) over the indices $dg$ in this case
yields
$$\textstyle 0=
\frac12R_{ab}{}^{ef}R_{cdef}+R_a{}^{ef}{}_cR_{befd}-R_b{}^{ef}{}_cR_{aefd}.$$
Now, if $\omega_{ab}$ is a section of $K$, it follows that 
$$\textstyle 0=\frac12\|R_{ab}{}^{cd}\omega_{cd}\|^2
+2\|R_{fc[b}{}^d\omega_{e]d}\|^2$$
and, therefore, that $R_{ab}{}^{cd}\omega_{cd}=0$.  In particular, since
$R_{abcd}$ is a section of $K\otimes K$, we conclude that
$\|R_{abcd}\|^2=R_{ab}{}^{cd}R_{cd}{}^{ab}=0$ so $R_{abcd}=0$, as required.

\section{Products}\label{products}
We may now modify our proof of Lemma~\ref{algebraic_result_for_LTS}, and hence
of Theorem~\ref{main_thm} in the irreducible case, so that it applies to a
product
$$M=M_1\times M_2\times\cdots\times M_k,$$
of irreducible Riemannian locally symmetric spaces, none of which is flat.
Since the Ricci curvature of such a product is block diagonal, with each block
being the Ricci curvature of an individual factor, a simple way of saying that
there are no flat factors is to say that the Ricci tensor $R_{ab}$ of the whole
is nondegenerate. Thus, according to Proposition~\ref{prop2}, in order to
prove Theorem~\ref{main_thm} in this case, it suffices to establish the
following.

\begin{lemma}\label{really_good_lemma}
Suppose $X_{abc}$ is a section of $\Wedge^1\otimes C$ and satisfies
\eqref{trace}. If $R_{ab}$ is nondegenerate, then $X_{abc}=0$.
\end{lemma}
\begin{proof}
If we set $Y_{abc}\equiv R_{[a}{}^dX_{b]cd}$, then (\ref{trace}) reads
\begin{equation}\label{neuetrace}
2Y_{abe}=R_{ab}{}^{cd}X_{cde}-R_{ab}{}^d{}_eX^c{}_{cd}.\end{equation}
Notice that
$Y_{abc}\equiv R_{[a}{}^dX_{b]cd}\Rightarrow 
2Y_{ab}{}^b=R_a{}^dX^c{}_{cd}$ whereas (\ref{neuetrace}) implies
$$2Y_{ab}{}^b=R_a{}^{bcd}X_{cdb}-R_a{}^dX^c{}_{cd}.$$
As in our proof of Lemma~\ref{algebraic_result_for_LTS}, it follows that 
$R_a{}^dX^c{}_{cd}=0$ and, therefore, since $R_{ab}$ is nonsingular, that 
$X^c{}_{cd}=0$. Equation~(\ref{neuetrace}) now reads
\begin{equation}\label{twisted_eigenvalue_problem}
2Y_{abe}=R_{ab}{}^{cd}X_{cde}.\end{equation}
Let us write $S_{ab}$ for the inverse of $R_{ab}$ and define
$$S_{abcd}\equiv R_{abc}{}^eS_{de}.$$
Tracing (\ref{LTS}) over $dg$ shows that $S_{abcd}$ satisfies Riemann tensor
symmetries
$$S_{abcd}=S_{[ab][cd]}\quad\mbox{and}\quad S_{[abc]d}=0$$
and (\ref{twisted_eigenvalue_problem}) becomes
\begin{equation}\label{eigenvalue2} S_{ab}{}^{cd}Y_{cde}=2Y_{abe}
\end{equation}
for the symmetric endomorphism $S_{ab}{}^{cd}$ of $\Wedge^2$. An advantage
of~$S_{ab}{}^{cd}$ over $R_{ab}{}^{cd}$, however, is that this tensor does not
see a constant rescaling of the metric. In fact, all of the following
$$\nabla_a,\quad R_{ab}{}^c{}_d,\quad R_{ab},\quad S^{ab},\quad\mbox{and}\quad
S_{ab}{}^{cd}=R_{ab}{}^c{}_eS^{de}$$
are preserved under $g_{ab}\mapsto\mbox{constant}\times g_{ab}$. Consequently,
for products without flat factors, Theorem~\ref{eigenvalues_of_curvature} now
implies, factor-by-factor, that all eigenvalues of
$$S:\Wedge^2\to\Wedge^2\quad\mbox{given by}\quad
\omega_{ab}\longmapsto S_{ab}{}^{cd}\omega_{cd}$$
are real and lie in the range $[0,2]$. Furthermore, eigenvalue $2$ is only
attained on Hermitian factors, in which case $\omega_{ab}$ must be a multiple of
the K\"ahler form on each such factor. {From} (\ref{eigenvalue2}) and
$Y_{ab}{}^b=\frac12R_a{}^dX^c{}_{cd}=0$ it follows that $Y_{abc}=0$ and hence
that $X_{abc}=0$. \end{proof}

The prolongation connection (\ref{prolongation_connection}) not only controls
the range of the Killing operator, as in~(\ref{criterion}), but also its 
kernel, as follows.
\begin{prop}\label{prolongation_isomorphism} There is an isomorphism
$$\left\{\left[\!\begin{array}{c}\sigma_b\\ 
\mu_{bc}\end{array}\!\right]\in\Gamma(E)
\,\;\raisebox{-9pt}{\rule{.5pt}{25pt}}\;\,
\nabla_a\!\left[\!\begin{array}{c}\sigma_b\\ 
\mu_{bc}\end{array}\!\right]=0\right\}
\stackrel{\simeq\enskip\;}{\longrightarrow}
\{\sigma_b\in\Gamma(\Wedge^1)\mid\nabla_{(a}\sigma_{b)}=0\}.$$
\end{prop}
\begin{proof} Clearly, 
$$\nabla_{(a}\sigma_{b)}=0\enskip\iff\enskip
\nabla_a\sigma_b=\mu_{ab}\,,\enskip\mbox{for some}\enskip
\mu_{ab}\in\Gamma(\Wedge^2).$$
Hence, it suffices to show that, if $\nabla_a\sigma_b=\mu_{ab}$, then 
$\nabla_a\mu_{bc}=R_{bc}{}^d{}_a\sigma_d$. To see this, observe that 
$\nabla_{[a}\mu_{bc]}=\nabla_{[a}\nabla_b\sigma_{c]}=0$ so
$$\nabla_a\mu_{bc}=\nabla_c\mu_{ba}-\nabla_b\mu_{ca}
=\nabla_c\nabla_b\sigma_a-\nabla_b\nabla_c\sigma_a=R_{bc}{}^d{}_a\sigma_d,$$
as required.\end{proof} 
\noindent This proposition leads us to consider just the second line of the
prolongation connection, 
$$\nabla_a\mu_{bc}=R_{bc}{}^d{}_a\sigma_d,$$
noting that if $R_{ab}$ is non-singular then $\nabla^c\mu_{bc}=R_b{}^d\sigma_d$
shows that $\sigma_b$ is determined by $\mu_{bc}$. We ask if $\sigma_b$ is
then necessarily a Killing field. 
\begin{lemma}\label{backward_prolongation}
Suppose $M$ is a Riemannian locally symmetric space with neither 
Hermitian nor flat factors. If $\mu_{cd}$ is $2$-form on $M$ so that
$$\nabla_b\mu_{cd}=R_{cd}{}^e{}_b\sigma_e,\quad\mbox{for some (uniquely 
determined) $1$-form $\sigma_c$ on $M$},$$
then $\nabla_b\sigma_c=\mu_{bc}$.
\end{lemma}
\begin{proof} 
In terms of the prolongation connection, we are given that
$$\nabla_b\!\left[\!\begin{array}{c}\sigma_c\\ 
\mu_{cd}\end{array}\!\right]=
\left[\!\begin{array}{c}\phi_{bc}\\ 
0\end{array}\!\right],\quad\mbox{for some tensor}\enskip\phi_{bc},$$
and from its curvature (\ref{curvature}) we deduce that
$$\left[\!\begin{array}{c}\nabla_{[a}\phi_{b]c}\\
-R_{cd}{}^e{}_{[a}\phi_{b]e}\end{array}\!\right]
=\left[\!\begin{array}{c}0\\ 
R_{ab}{}^e{}_{[c}\mu_{d]e}+R_{cd}{}^e{}_{[a}\mu_{b]e}
\end{array}\!\right]$$
In particular, we see that $R_{cd}{}^e{}_{[a}\phi_{b]e}$ is in the range of 
${\mathcal{R}}:\Wedge^2\to\Wedge^2\otimes\Wedge^2$ and, therefore, in the 
range of ${\mathcal{R}}:C\to\Wedge^2\otimes\Wedge^2$. Hence,
$$R_{cd}{}^e{}_{[a}\phi_{b]e}-R_{cd}{}^e{}_{[a}\omega_{b]e}
=R_{ab}{}^e{}_{[c}\omega_{d]e},\quad\mbox{for some}\enskip
\omega_{ab}\in\Gamma(C).$$
However, the tensor on the right hand side of this equation lies in 
$K\otimes C$, whereas the tensor on the left lies in~$\Wedge^2\otimes K$. 
Therefore, both sides vanish and, in particular, we deduce that
\begin{equation}\label{R_kills_phi}
R_{cd}{}^e{}_{[a}\phi_{b]e}=0.\end{equation}
Tracing this equation over $bc$ gives
\begin{equation}\label{tracing_this_equation}
R_d{}^e\phi_{ae}=R_d{}^{be}{}_a\phi_{be}.\end{equation} 
Now recall, as in our proof of Lemma~\ref{really_good_lemma}, that if we write
$S^{ab}$ for the inverse of~$R_{ab}$, then $S_{abcd}\equiv R_{abc}{}^eS_{de}$
satisfies Riemann tensor symmetries. Thus, applying $S_f{}^d$ to
(\ref{tracing_this_equation}) gives
$$\phi_{af}=S_f{}^dR_d{}^{be}{}_a\phi_{be}=S_f{}^{be}{}_a\phi_{be},$$
which decomposes into symmetric and skew parts
$$\textstyle\phi_{(af)}=S_f{}^{be}{}_a\phi_{(be)}\quad\mbox{and}\quad
\phi_{[af]}=S_{[f}{}^{be}{}_{a]}\phi_{be}=\frac12S_{af}{}^{be}\phi_{be},$$
this final equality from the Bianchi symmetry in the form
$S_{[f}{}^{be}{}_{a]}=\frac12S_{af}{}^{be}$.  However, as observed in the proof
of Lemma~\ref{really_good_lemma}, since $M$ has no Hermitian factors, the
eigenvalues of $S_{ab}{}^{cd}$ lie in the interval $[0,2)$.  It follows that
$\phi_{ab}$ is symmetric.  Let $\psi_{ab}\equiv \phi_a{}^c\phi_{bc}$.  Since
$\phi_{ab}$ is symmetric so is $\psi_{ab}$ and, from (\ref{R_kills_phi}), we
find that
\begin{equation}\label{killed_by_R}
R_{ab}{}^e{}_{(c}\psi_{d)e}=0.\end{equation}
When the Ricci tensor is nondegenerate, a $1$-form annihilated by the curvature 
necessarily vanishes:
\begin{equation}\label{hit_and_skew}
R_{ab}{}^d{}_c\theta_d=0 \enskip\Rightarrow\enskip
R_a{}^d\theta_d=0\enskip\Rightarrow\enskip \theta_a=0.\end{equation}
This implies that a symmetric $2$-form on a Riemannian product with
nondegenerate Ricci tensor and satisfying (\ref{killed_by_R}) can have no
`cross terms,' i.e.~must be block diagonal.  
Therefore, without loss of generality, we may suppose for the rest of this
proof that $M$ is irreducible and that $R_{ab}=\pm g_{ab}$, as usual.  Having
done this, equation (\ref{killed_by_R}) says that the natural action of
curvature annihilates the symmetric form $\psi_{de}$ and since $R_{ab}{}^e{}_c$
is covariantly constant the same is true for~$\nabla_a\psi_{de}$.  As in the
proof of Theorem~\ref{eigenvalues_of_curvature}, now using Lemma~\ref{L-one},
this implies that $\psi_{ab}=\lambda g_{ab}$ for some smooth
function~$\lambda\geq 0$ and it remains to show that $\lambda\equiv 0$.  
Let us record our conclusion so far
\begin{equation}\label{no_loss}
\phi_a{}^c\phi_{bc}=\lambda g_{ab}\end{equation}
and now deal with the case $R_{ab}=g_{ab}$. 
Equation (\ref{tracing_this_equation}) becomes
$R_{abcd}\phi^{bc}=\phi_{ad}$ and (\ref{LTS_traced}) therefore yields
\begin{equation}\label{yield}4\phi_{ad}\phi^{ad}=
\phi^{ad}\phi^{bc}R_{ab}{}^{ef}R_{cdef}
+2\phi^{ad}\phi^{bc}R_a{}^{ef}{}_cR_{befd}.\end{equation}
But (\ref{R_kills_phi}) also implies, with~(\ref{no_loss}), that
$$\phi^{ad}\phi^{bc}R_{ab}{}^{ef}R_{cdef}
=\phi^{ad}\phi^b{}_aR{}^c{}_b{}^{ef}R_{cdef}
=\lambda g^{bd}R{}^c{}_b{}^{ef}R_{cdef}
=\lambda\|R_{abcd}\|^2.$$
We also know that $\phi_{ad}\phi^{ad}=\lambda\delta_d{}^d=\lambda n$, where $n$
is the dimension of $M$, so (\ref{yield}) becomes
$$4\lambda n=\lambda\|R_{abcd}\|^2
+2\phi^{ad}\phi^{bc}R_a{}^{ef}{}_cR_{befd}.$$
To sort out the last term, again we use (\ref{R_kills_phi}) to 
conclude that 
$$0=R^{cdf}{}_{[a}\phi_{b]f}R^{abe}{}_{[c}\phi_{d]e}
=R^{cdf}{}_a\phi_{bf}R^{abe}{}_c\phi_{de}
=\phi^{ad}\phi^{bc}R_a{}^{ef}{}_cR_{bfed}$$
and, therefore, 
$$2\phi^{ad}\phi^{bc}R_a{}^{ef}{}_cR_{befd}
=4\phi^{ad}\phi^{bc}R_a{}^{ef}{}_cR_{b[ef]d}
=2\phi^{ad}\phi^{bc}R_a{}^{ef}{}_cR_{dbef}
=\phi^{ad}\phi^{bc}R_{ca}{}^{ef}R_{dbef},$$
which we have already found to be $\lambda\|R_{abcd}\|^2$.  We conclude that
$4\lambda n=2\lambda\|R_{abcd}\|^2$ and hence, unless $\lambda\equiv 0$, that
$\|R_{abcd}\|^2=2n$.  We also know that $R_{ab}{}^{ab}=\delta_b{}^b=n$.  In
summary, for the endomorphism $R_{ab}{}^{cd}$ of~$\Wedge^2$, we have found that
\begin{itemize}
\item all eigenvalues are in the range $[0,2)$;
\item the sum of the eigenvalues is $n$;
\item the sum of the squares of the eigenvalues is $2n$.
\end{itemize}
This is a contradiction. The case $R_{ab}=-g_{ab}$ follows in a similar
fashion.
\end{proof}
To complete the proof Theorem~\ref{main_thm}, it remains to consider Riemannian
locally symmetric products of the form
\begin{equation}\label{final_case}M\times{\mathbb{R}}^p,\end{equation}
where $M$ has no flat factors and ${\mathbb{R}}^p$ is equipped with its
standard metric.  We need to show that if $M$ has no Hermitian factors, then
(\ref{key_complex}) is locally exact but, if $M$ has a Hermitian factor,
then (\ref{key_complex}) fails to be locally exact.

To approach these final cases, let us simply write out the complex
(\ref{key_complex}) on a Riemannian product $M\times{\mathbb{R}}^p$.  For this
purpose we may borrow some notation from complex geometry and write
$$\Wedge^1=\Wedge^{1,0}\oplus\Wedge^{0,1},$$
where $\Wedge^{1,0}$ denotes the pull-back to $M\times{\mathbb{R}}^p$ of the
cotangent bundle on $M$ and $\Wedge^{0,1}$ denotes the pull-back of the
cotangent bundle on~${\mathbb{R}}^p$. We obtain an induced splitting of
symmetric forms on $M\times{\mathbb{R}}^p$ according to `type:'
\begin{equation}\label{symmetric_two-forms_according_to_type}
\textstyle\bigodot^2\!\Wedge^1
=\bigodot^2\!\Wedge^{1,0}\oplus\Wedge^{1,1}\oplus\bigodot^2\!\Wedge^{0,1},
\quad\mbox{where}\enskip \Wedge^{1,1}=\Wedge^{1,0}\otimes\Wedge^{0,1}.
\end{equation}
If we also also use the `barred and unbarred' indices from complex 
geometry, the Killing operator becomes
\begin{equation}\label{Killing_on_product_manifold}
\left[\!\begin{array}{c} X_b\\ \xi_{\bar{b}}\end{array}\!\right]
\stackrel{\mathcal{K}}{\longmapsto}\left[\!\begin{array}{c}
\nabla_{(a}X_{b)}\\ \nabla_a\xi_{\bar{a}}+\partial_{\bar{a}}X_a\\
\partial_{(\bar{a}}\xi_{\bar{b})}
\end{array}\!\right],\end{equation}
where $\nabla_a$ is the metric connection on $M$ and $\partial_{\bar{a}}$ is
the standard co\"ordinate derivative on~${\mathbb{R}}^p$, both pulled back to
$M\times{\mathbb{R}}^p$ in the obvious way. Notice that
$\nabla_a\partial_{\bar{a}}=\partial_{\bar{a}}\nabla_a$ acting upon any tensor
field. The Calabi operator ${\mathcal{C}}:
\begin{picture}(12,6)(0,0)
\put(0,0){\line(1,0){12}}
\put(0,6){\line(1,0){12}}
\put(0,0){\line(0,1){6}}
\put(6,0){\line(0,1){6}}
\put(12,0){\line(0,1){6}}
\end{picture}\longrightarrow
\begin{picture}(12,12)(0,2)
\put(0,0){\line(1,0){12}}
\put(0,6){\line(1,0){12}}
\put(0,12){\line(1,0){12}}
\put(0,0){\line(0,1){12}}
\put(6,0){\line(0,1){12}}
\put(12,0){\line(0,1){12}}
\end{picture}$\enskip breaks up into irreducibles
\begin{equation}\label{calabi_on_product}\begin{array}{r}
\begin{picture}(12,6)(0,0)
\put(0,0){\line(1,0){12}}
\put(0,6){\line(1,0){12}}
\put(0,0){\line(0,1){6}}
\put(6,0){\line(0,1){6}}
\put(12,0){\line(0,1){6}}
\end{picture}\,\Wedge^{1,0}\\[14pt]
\Wedge^{1,0}\otimes\Wedge^{0,1}\\[14pt]
\begin{picture}(12,6)(0,0)
\put(0,0){\line(1,0){12}}
\put(0,6){\line(1,0){12}}
\put(0,0){\line(0,1){6}}
\put(6,0){\line(0,1){6}}
\put(12,0){\line(0,1){6}}
\end{picture}\,\Wedge^{0,1}
\end{array}
\begin{picture}(82,0)
\qbezier (0,33) (0,33) (80,53) 
\qbezier (0,32) (0,32) (80,32)
\qbezier (0,31) (0,31) (80,12)
\qbezier (0,5) (0,5) (80,31)
\qbezier (0,4) (0,4) (80,11)
\qbezier (0,3) (0,3) (80,-5)
\qbezier (0,2) (0,2) (80,-26)
\qbezier (0,-26) (0,-26) (80,10)
\qbezier (0,-27) (0,-27) (80,-27)
\qbezier (0,-28) (0,-28) (80,-48) 
\end{picture}
\begin{array}{l}
\begin{picture}(12,12)(0,2)
\put(0,0){\line(1,0){12}}
\put(0,6){\line(1,0){12}}
\put(0,12){\line(1,0){12}}
\put(0,0){\line(0,1){12}}
\put(6,0){\line(0,1){12}}
\put(12,0){\line(0,1){12}}
\end{picture}\,\Wedge^{1,0}\\[5pt]
\begin{picture}(12,12)(0,2)
\put(0,0){\line(1,0){6}}
\put(0,6){\line(1,0){12}}
\put(0,12){\line(1,0){12}}
\put(0,0){\line(0,1){12}}
\put(6,0){\line(0,1){12}}
\put(12,6){\line(0,1){6}}
\end{picture}\,\Wedge^{1,0}\otimes\Wedge^{0,1}\\[5pt]
\begin{picture}(12,6)(0,0)
\put(0,0){\line(1,0){12}}
\put(0,6){\line(1,0){12}}
\put(0,0){\line(0,1){6}}
\put(6,0){\line(0,1){6}}
\put(12,0){\line(0,1){6}}
\end{picture}\,\Wedge^{1,0}
\otimes
\begin{picture}(12,6)(0,0)
\put(0,0){\line(1,0){12}}
\put(0,6){\line(1,0){12}}
\put(0,0){\line(0,1){6}}
\put(6,0){\line(0,1){6}}
\put(12,0){\line(0,1){6}}
\end{picture}\,\Wedge^{0,1}\\[5pt]
\Wedge^{2,0}\otimes\Wedge^{0,2}\\[5pt]
\Wedge^{1,0}\otimes\begin{picture}(12,12)(0,2)
\put(0,0){\line(1,0){6}}
\put(0,6){\line(1,0){12}}
\put(0,12){\line(1,0){12}}
\put(0,0){\line(0,1){12}}
\put(6,0){\line(0,1){12}}
\put(12,6){\line(0,1){6}}
\end{picture}\,\Wedge^{0,1}\\[5pt]
\begin{picture}(12,12)(0,2)
\put(0,0){\line(1,0){12}}
\put(0,6){\line(1,0){12}}
\put(0,12){\line(1,0){12}}
\put(0,0){\line(0,1){12}}
\put(6,0){\line(0,1){12}}
\put(12,0){\line(0,1){12}}
\end{picture}\,\Wedge^{0,1}
\end{array}
\end{equation}
where, for example, $\begin{picture}(12,6)(0,0)
\put(0,0){\line(1,0){12}}
\put(0,6){\line(1,0){12}}
\put(0,0){\line(0,1){6}}
\put(6,0){\line(0,1){6}}
\put(12,0){\line(0,1){6}}
\end{picture}\,\Wedge^{1,0}$ denotes the symmetric tensor product functor 
applied to~$\Wedge^{1,0}$.
In particular, there are four parts to ${\mathcal{C}}$ applied to
$\Wedge^{1,0}\otimes\Wedge^{0,1}$, specifically 
\begin{equation}\label{four_parts}\begin{array}{l}
\nabla_c\nabla_{[a}h_{b]}{}^{\bar{b}}-\frac12R_{ab}{}^d{}_ch_d{}^{\bar b}
\in\Gamma\big(\,\begin{picture}(12,12)(0,2)
\put(0,0){\line(1,0){6}}
\put(0,6){\line(1,0){12}}
\put(0,12){\line(1,0){12}}
\put(0,0){\line(0,1){12}}
\put(6,0){\line(0,1){12}}
\put(12,6){\line(0,1){6}}
\end{picture}\,\Wedge^{1,0}\otimes\Wedge^{0,1}\big),\\[4pt]
\nabla_{(a}\partial^{(\bar{a}}h_{b)}{}^{\bar{b})}
\in\Gamma\big(\begin{picture}(12,6)(0,0)
\put(0,0){\line(1,0){12}}
\put(0,6){\line(1,0){12}}
\put(0,0){\line(0,1){6}}
\put(6,0){\line(0,1){6}}
\put(12,0){\line(0,1){6}}
\end{picture}\,\Wedge^{1,0}
\otimes
\begin{picture}(12,6)(0,0)
\put(0,0){\line(1,0){12}}
\put(0,6){\line(1,0){12}}
\put(0,0){\line(0,1){6}}
\put(6,0){\line(0,1){6}}
\put(12,0){\line(0,1){6}}
\end{picture}\,\Wedge^{0,1}\big),\\[4pt]
\nabla_{[a}\partial^{[\bar{a}}h_{b]}{}^{\bar{b}]}
\in\Gamma\big(\Wedge^{2,0}\otimes\Wedge^{0,2}\big),\\[4pt]
\partial^{\bar{c}}\partial^{[\bar{a}}h_b{}^{\bar{b}]}
\in\Gamma\big(\Wedge^{1,0}\otimes\begin{picture}(12,12)(0,2)
\put(0,0){\line(1,0){6}}
\put(0,6){\line(1,0){12}}
\put(0,12){\line(1,0){12}}
\put(0,0){\line(0,1){12}}
\put(6,0){\line(0,1){12}}
\put(12,6){\line(0,1){6}}
\end{picture}\,\Wedge^{0,1}\big),
\end{array}\end{equation}
where we have used (\ref{alternative_calabi}) to obtain the first operator.
Only this first operator is quotiented by ${\mathcal{R}}:\Wedge^2\to
\begin{picture}(12,12)(0,2)
\put(0,0){\line(1,0){12}}
\put(0,6){\line(1,0){12}}
\put(0,12){\line(1,0){12}}
\put(0,0){\line(0,1){12}}
\put(6,0){\line(0,1){12}}
\put(12,0){\line(0,1){12}}
\end{picture}$ in passing to the operator ${\mathcal{L}}$ 
in~(\ref{key_complex}), specifically to obtain the compatibility condition
\begin{equation}\label{compatibility_condition}
\nabla_c\nabla_{[a}h_{b]\bar{b}}=R_{ab}{}^d{}_c\kappa_{d\bar{b}}
\quad\mbox{for some}\enskip \kappa_{b\bar{b}}\in\Gamma(\Wedge^{1,1}).
\end{equation}
\begin{prop}
Suppose $M$ is locally symmetric with nondegenerate Ricci tensor but having a
Hermitian factor.  Then \eqref{key_complex} fails to be locally exact on the
Riemannian product $M\times{\mathbb{R}}^p$, where ${\mathbb{R}}^p$ has the
standard flat metric.
\end{prop}
\begin{proof} Choose a non-trivial covariantly constant $1$-form
$\theta_{\bar{b}}$ on ${\mathbb{R}}^p$ and on $M$ let us choose a $2$-form
$J_{ab}$ by pulling back the K\"ahler form from a Hermitian factor.  Being
closed, we may locally choose a $1$-form $\phi_a$ on $M$ such that
$\nabla_{[a}\phi_{b]}=J_{ab}$.  Let us use the same notation $\phi_a$ for the
pullback of this form to $M\times{\mathbb{R}}^p$.  Also pull back
$\theta_{\bar{b}}$ to the product and consider
$h_{b\bar{b}}\equiv\phi_b\theta_{\bar{b}}\in\Gamma(\Wedge^{1,1})$ as a
symmetric $2$-form there.  As $\partial_{\bar{a}}\theta_{\bar{b}}=0$, the last
three operators from (\ref{four_parts}) annihilate~$h_{b\bar{b}}$.  Also, as
$\nabla_cJ_{ab}=0$, the compatibility condition (\ref{compatibility_condition})
holds with~$\kappa_{b\bar b}=0$.  But $h_{b\bar{b}}$ cannot be written as
$$\nabla_b\xi_{\bar{b}}+\partial_{\bar{b}}X_b\quad\mbox{for}\enskip
X_b\in\Gamma(\Wedge^{1,0})\enskip\mbox{such that}\enskip
\nabla_{(a}X_{b)}=0\enskip\mbox{and}\enskip
\xi_{\bar{b}}\in\Gamma(\Wedge^{0,1})$$
as would be required by (\ref{Killing_on_product_manifold}) to be in the range 
of~${\mathcal{K}}$ since, if this were the case, then  
$$0=\nabla_{[a}\nabla_{b]}\xi_{\bar{b}}
=\nabla_{[a}h_{b]\bar{b}}-\partial_{\bar{b}}\nabla_{[a}X_{b]}\quad\Rightarrow
\quad\theta_{\bar{b}}J_{ab}=\partial_{\bar{b}}\nabla_{[a}X_{b]}
=\partial_{\bar{b}}\nabla_aX_b.$$
Then, as $J_{ab}$ is covariant constant, we now find that 
$$\textstyle 0=\theta_{\bar{b}}\nabla_{[c}J_{a]b}
=\partial_{\bar{b}}\nabla_{[c}\nabla_{a]}X_b=
\frac12R_{ac}{}^d{}_b\partial_{\bar{b}}X_d$$
and hence that $\partial_{\bar{b}}X_d=0$. It follows that $J_{ab}=0$, a 
contradiction.
\end{proof}
The following proposition completes the proof of Theorem~\ref{main_thm}.
\begin{prop} Suppose $M$ is a Riemannian locally symmetric space with
nondegenerate Ricci tensor and no Hermitian factors. Then \eqref{key_complex}
is locally exact on the Riemannian product $M\times{\mathbb{R}}^p$, where
${\mathbb{R}}^p$ has the standard flat metric.
\end{prop}
\begin{proof}
We already know that (\ref{key_complex}) is locally exact on $M$ and
on~${\mathbb{R}}^p$. Therefore, looking at~(\ref{Killing_on_product_manifold}),
(\ref{calabi_on_product}), (\ref{four_parts}), and
(\ref{compatibility_condition}), we are required to show that if
$h_{b\bar{b}}\in\Gamma(\Wedge^{1,1})$ satisfies
\begin{equation}\label{have}
\nabla_c\nabla_{[a}h_{b]\bar{b}}=R_{ab}{}^d{}_c\kappa_{d\bar{b}},\enskip
\nabla_{(a}\partial^{(\bar{a}}h_{b)}{}^{\bar{b})}=0,\enskip
\nabla_{[a}\partial^{[\bar{a}}h_{b]}{}^{\bar{b}]}=0,\enskip\mbox{and}\enskip
\partial^{\bar{c}}\partial^{[\bar{a}}h_b{}^{\bar{b}]}=0,\end{equation} then
locally we may find $X_b\in\Gamma(\Wedge^{1,0})$ and
$\xi_{\bar{b}}\in\Gamma(\Wedge^{0,1})$ such that
\begin{equation}\label{want}\nabla_{(a}X_{b)}=0,\quad
\nabla_a\xi_{\bar{a}}+\partial_{\bar{a}}X_a=h_{a\bar{a}},\quad\mbox{and}\quad
\partial_{(\bar{a}}\xi_{\bar{b})}=0.\end{equation} 
{From} the first equation of (\ref{have}), Lemma~\ref{backward_prolongation}
tells us that $\nabla_a\kappa_{b\bar{b}}=\nabla_{[a}h_{b]\bar{b}}$ and the
third equation from (\ref{have}) implies that
$\nabla_a\partial^{[\bar{a}}\kappa_b{}^{\bar{b}]}=0$. 
Therefore $R_{ab}{}^c{}_d\partial^{[\bar{a}}\kappa_c{}^{\bar{b}]}=0$. 
From~(\ref{hit_and_skew}), it follows that
$\partial^{[\bar{a}}\kappa_b{}^{\bar{b}]}=0$. Thus, as a closed $1$-form along
the fibres of $M\times{\mathbb{R}}^p\to M$, we may integrate to find $X_b$ such
that $\partial_{\bar{b}}X_b=\kappa_{b\bar{b}}$ and, by differentiating under
the integral sign, it follows that $\nabla_{(a}X_{b)}=0$, which is the first 
requirement of~(\ref{want}). If we
introduce $\psi_{a\bar{a}}\equiv h_{a\bar{a}}-\kappa_{a\bar{a}}$, then 
\begin{equation}\label{psi_closed}
\nabla_a\psi_{b\bar{b}}=\nabla_ah_{b\bar{b}}-\nabla_a\kappa_{b\bar{b}}
=\nabla_ah_{b\bar{b}}-\nabla_{[a}h_{b]\bar{b}}
=\nabla_{(a}h_{b)\bar{b}}\end{equation}
so
$$\nabla_a\partial^{\bar{a}}\psi_b{}^{\bar{b}}
=\nabla_{(a}\partial^{\bar{a}}h_{b)}{}^{\bar{b}}
=\nabla_{(a}\partial^{[\bar{a}}h_{b)}{}^{\bar{b}]},$$
this last equality from the second equation of~(\ref{have}). Using, once more,
the third equation from (\ref{have}), we now have
$$\nabla_a\partial^{\bar{a}}\psi_b{}^{\bar{b}}
=\nabla_{(a}\partial^{[\bar{a}}h_{b)}{}^{\bar{b}]}
+\nabla_{[a}\partial^{[\bar{a}}h_{b]}{}^{\bar{b}]}
=\nabla_a\partial^{[\bar{a}}h_b{}^{\bar{b}]}$$
and (\ref{hit_and_skew}) implies that
\begin{equation}\label{look_it_is_skew}
\partial^{\bar{a}}\psi_b{}^{\bar{b}}=\partial^{[\bar{a}}h_b{}^{\bar{b}]}.
\end{equation}
Having found~$X_b$ and introduced 
$\psi_{a\bar{a}}=h_{a\bar{a}}-\kappa_{a\bar{a}}$, in order to
satisfy (\ref{want}) it remains to find $\xi_{\bar{b}}$, such that
$$\nabla_b\xi_{\bar{b}}=\psi_{b\bar{b}}\quad\mbox{and}\quad
\partial_{(\bar{a}}\xi_{\bar{b})}=0.$$
{From} (\ref{psi_closed}) and (\ref{look_it_is_skew}) it follows that 
$$\nabla_{[a}\psi_{b]\bar{b}}=0\quad\mbox{and}\quad
\partial^{(\bar{a}}\psi_b{}^{\bar{b})}=0$$
so $\xi_{\bar{b}}$ can be obtained by integrating along the fibres of 
$M\times{\mathbb{R}}^p\to{\mathbb{R}}^p$.
\end{proof}
Notice that, in this proof, we apparently omitted to use the last constraint of
(\ref{have}), namely that
$\partial^{\bar{c}}\partial^{[\bar{a}}h_b{}^{\bar{b}]}=0$. In fact, as in the
proof of Proposition~\ref{prolongation_isomorphism}, this constraint follows by
prolongation from (\ref{look_it_is_skew}) so nothing is lost by this apparent
omission.

\section{Concluding remarks}
\subsection{Some key examples} By design, Theorem~\ref{main_thm} includes the 
round sphere~$S^n$. In this case, the homomorphism ${\mathcal{R}}$ in 
(\ref{key_hom}) vanishes, whence $K=\Wedge^2$ and (\ref{calabi}) reads
$$h_{ab}\mapsto(\nabla_{(a}\nabla_{c)}+g_{ac})h_{bd}
              -(\nabla_{(b}\nabla_{c)}+g_{bc})h_{ad}
              -(\nabla_{(a}\nabla_{d)}+g_{ad})h_{bc}
              +(\nabla_{(b}\nabla_{d)}+g_{bd})h_{ac}$$
on the {\em unit\/} sphere (where $R_{abcd}=g_{ac}g_{bd}-g_{bc}g_{ad}$), as 
may be gleaned from the {\em Lagrangian atlases\/} of~\cite{C} or found 
explicitly as~\cite[Th\'eor\`eme~6.1]{GG}.

For a product of round spheres $S^p\times S^q$, we may decompose bundles
according to type as in~\S\ref{products}, and it is easy to check that
$K=\Wedge^{2,0}\oplus\Wedge^{0,2}$ and, therefore, that $C=\Wedge^{1,1}$.
Theorem~\ref{main_thm} says that (\ref{key_complex}) is exact on all these
products save for $S^2\times S^1$ (and $S^1\times S^2$).

For complex projective space ${\mathbb{CP}}_n$ with its Fubini-Study metric 
and K\"ahler form $J_{ab}$, 
we find that
$$K=\{\mu_{cd}\mid J_a{}^cJ_b{}^d\mu_{cd}=\mu_{ab}\}\quad\mbox{and}\quad
C=\{\mu_{cd}\mid J_a{}^cJ_b{}^d\mu_{cd}=-\mu_{ab}\},$$
bundles with familiar complexifications, namely $\Wedge^{1,1}$ and 
$\Wedge^{2,0}\oplus\Wedge^{0,2}$, respectively. As an 
${\mathrm{SU}}(n)$-bundle, the image 
$C\hookrightarrow\begin{picture}(12,12)(0,2)
\put(0,0){\line(1,0){12}}
\put(0,6){\line(1,0){12}}
\put(0,12){\line(1,0){12}}
\put(0,0){\line(0,1){12}}
\put(6,0){\line(0,1){12}}
\put(12,0){\line(0,1){12}}
\end{picture}$\enskip is irreducible (it is ${\mathcal{W}}_9$ in the full
decomposition of~$\:\begin{picture}(12,12)(0,2)
\put(0,0){\line(1,0){12}}
\put(0,6){\line(1,0){12}}
\put(0,12){\line(1,0){12}}
\put(0,0){\line(0,1){12}}
\put(6,0){\line(0,1){12}}
\put(12,0){\line(0,1){12}}
\end{picture}$\enskip into 10 irreducible bundles given in~\cite{TV}). This 
raises the possibility of removing further ${\mathrm{SU}}(n)$-subbundles 
from~$\:\begin{picture}(12,12)(0,2)
\put(0,0){\line(1,0){12}}
\put(0,6){\line(1,0){12}}
\put(0,12){\line(1,0){12}}
\put(0,0){\line(0,1){12}}
\put(6,0){\line(0,1){12}}
\put(12,0){\line(0,1){12}}
\end{picture}$\enskip whilst keeping local exactness of the resulting complex. 
One such option has already been considered in~\cite{ES}, specifically the
complex
\begin{equation}\label{symplectic_complex}
\Wedge^1=\begin{picture}(6,6)(0,0)
\put(0,0){\line(1,0){6}}
\put(0,6){\line(1,0){6}}
\put(0,0){\line(0,1){6}}
\put(6,0){\line(0,1){6}}
\end{picture}\stackrel{\mathcal{K}}{\longrightarrow}
\begin{picture}(12,6)(0,0)
\put(0,0){\line(1,0){12}}
\put(0,6){\line(1,0){12}}
\put(0,0){\line(0,1){6}}
\put(6,0){\line(0,1){6}}
\put(12,0){\line(0,1){6}}
\end{picture}\longrightarrow
\begin{picture}(12,12)(0,2)
\put(0,0){\line(1,0){12}}
\put(0,6){\line(1,0){12}}
\put(0,12){\line(1,0){12}}
\put(0,0){\line(0,1){12}}
\put(6,0){\line(0,1){12}}
\put(12,0){\line(0,1){12}}
\end{picture}\raisebox{-3.5pt}{$\scriptstyle\perp$},
\end{equation}
where this last bundle comprises Riemann tensors that are trace-free with 
respect to $J_{ab}$.

In \cite{EGold} it is shown that (\ref{symplectic_complex}) is globally exact 
on~${\mathbb{CP}}_n$, whilst in~\cite{ES} it is shown that the local 
cohomology is a constant sheaf with fibre ${\mathfrak{su}}(n+1)$. In fact, we 
have removed three ${\mathrm{SU}}(n)$-irreducibles to obtain 
(\ref{symplectic_complex}):
$$\begin{picture}(12,12)(0,2)
\put(0,0){\line(1,0){12}}
\put(0,6){\line(1,0){12}}
\put(0,12){\line(1,0){12}}
\put(0,0){\line(0,1){12}}
\put(6,0){\line(0,1){12}}
\put(12,0){\line(0,1){12}}
\end{picture}
=\begin{picture}(12,12)(0,2)
\put(0,0){\line(1,0){12}}
\put(0,6){\line(1,0){12}}
\put(0,12){\line(1,0){12}}
\put(0,0){\line(0,1){12}}
\put(6,0){\line(0,1){12}}
\put(12,0){\line(0,1){12}}
\end{picture}\raisebox{-3.5pt}{$\scriptstyle\perp$}\oplus{\mathcal{R}}(C)
\oplus\Wedge_\perp^{1,1}\oplus\Wedge^0.$$
Replacing just $\Wedge^0$ is sufficient to restore local exactness as follows.
\begin{thm}
On complex projective space with its Fubini-Study metric, the complex
$$\Wedge^1=\begin{picture}(6,6)(0,0)
\put(0,0){\line(1,0){6}}
\put(0,6){\line(1,0){6}}
\put(0,0){\line(0,1){6}}
\put(6,0){\line(0,1){6}}
\end{picture}\stackrel{\mathcal{K}}{\longrightarrow}
\begin{picture}(12,6)(0,0)
\put(0,0){\line(1,0){12}}
\put(0,6){\line(1,0){12}}
\put(0,0){\line(0,1){6}}
\put(6,0){\line(0,1){6}}
\put(12,0){\line(0,1){6}}
\end{picture}\longrightarrow
\begin{picture}(12,12)(0,2)
\put(0,0){\line(1,0){12}}
\put(0,6){\line(1,0){12}}
\put(0,12){\line(1,0){12}}
\put(0,0){\line(0,1){12}}
\put(6,0){\line(0,1){12}}
\put(12,0){\line(0,1){12}}
\end{picture}\raisebox{-3.5pt}{$\scriptstyle\perp$}\oplus\Wedge^0$$
is locally exact.\end{thm}
\begin{proof}
In~\cite{ES} it is shown that the local kernel 
of\enskip$\begin{picture}(12,6)(0,0)
\put(0,0){\line(1,0){12}}
\put(0,6){\line(1,0){12}}
\put(0,0){\line(0,1){6}}
\put(6,0){\line(0,1){6}}
\put(12,0){\line(0,1){6}}
\end{picture}\to
\begin{picture}(12,12)(0,2)
\put(0,0){\line(1,0){12}}
\put(0,6){\line(1,0){12}}
\put(0,12){\line(1,0){12}}
\put(0,0){\line(0,1){12}}
\put(6,0){\line(0,1){12}}
\put(12,0){\line(0,1){12}}
\end{picture}\raisebox{-3.5pt}{$\scriptstyle\perp$}$ comprises tensors
of the form
$$h_{ab}=\nabla_{(a}X_{b)}+Y_{(a}\phi_{b)},$$
where $\nabla_{(a}Y_{b)}=0$ and $\nabla_{[a}\phi_{b]}=J_{ab}$. 
Looking ahead to \S\ref{warped_products}, 
for\enskip$\begin{picture}(12,6)(0,0)
\put(0,0){\line(1,0){12}}
\put(0,6){\line(1,0){12}}
\put(0,0){\line(0,1){6}}
\put(6,0){\line(0,1){6}}
\put(12,0){\line(0,1){6}}
\end{picture}\to\Wedge^0$, we may as well take the differential operator 
(\ref{the_composition}) and compute that, in this case,
$$\nabla_{(a}X_{b)}+Y_{(a}\phi_{b)}\longmapsto 6J^{ab}\nabla_aY_b.$$
Hence, being in the kernel of this operator forces $J^{bc}\nabla_bY_c=0$. 
{From} Proposition~\ref{prolongation_isomorphism}, it follows that, if we set
$\mu_{ab}=\nabla_aY_b$, then $\nabla_a\mu_{bc}=R_{bc}{}^d{}_aY_d$ and now
$$0=\nabla_a(J^{bc}\mu_{bc})=J^{bc}\nabla_a\mu_{bc}=J^{bc}R_{bc}{}^d{}_aY_d
=2J{}^d{}_aY_d.$$
We conclude that  $Y_a=0$ and the proof is complete.
\end{proof}

\subsection{The Riemann curvature operator}
Critical to our proof of Theorem~\ref{main_thm} was some control of the
eigenvalues of the Riemann curvature tensor when viewed as a symmetric
endomorphism (\ref{R}) of the $2$-forms. It is worthwhile noting what are 
these eigenvalues for the sphere and complex projective space. The 
normalisation $R_{ab}=g_{ab}$ implies
\begin{itemize}
\item for $S^n$:\quad $R_{abcd}=\frac1{n-1}(g_{ac}g_{bd}-g_{bc}g_{ad})$
\item for ${\mathbb{CP}}_n$:\quad $R_{abcd}=
\frac1{2(n+1)}
(g_{ac}g_{bd}-g_{bc}g_{ad}+J_{ac}J_{bd}-J_{bc}J_{ad}+2J_{ab}J_{cd}),$ 
\end{itemize}
where $J_{ab}$ is the K\"ahler form on ${\mathbb{CP}}_n$. Thus, on $S^n$
there is just one eigenspace, having eigenvalue $2/(n-1)$. On
${\mathbb{CP}}_n$, we find
\begin{itemize}
\item $\Wedge^{2,0}\oplus\Wedge^{0,2}$\enskip with eigenvalue $0$,
\item $\Wedge_\perp^{1,1}$\enskip with eigenvalue $2/(n+1)$,
\item $\langle J_{ab}\rangle$\enskip with eigenvalue $2$.
\end{itemize}
Complex projective space provides a good illustration of
Theorem~\ref{eigenvalues_of_curvature}.  It is essential that we are in the
Riemannian setting in order to conclude that all eigenvalues of
$\omega_{ab}\mapsto R_{ab}{}^{cd}\omega_{cd}$ are real.  This and several other
steps break down in the Lorentzian setting.  For example, it is no longer true
that a Ricci flat locally symmetric space need be flat.
We return to the Lorentzian case in a separate article~\cite{CELM2}.

\subsection{Warped products}\label{warped_products}
The methods in this article also apply to metrics beyond the locally symmetric
realm.  For example, following Khavkine~\cite{K}, we may ask about the range of
the Killing operator on Friedmann-Lemaitre-Robertson-Walker (FLRW) metrics,
namely {\em warped products\/} 
\begin{equation}\label{warped}
\Omega^2(t)g_{ab}\pm dt^2\quad\mbox{on}\enskip M\times{\mathbb{R}},
\end{equation}
where $\Omega:{\mathbb{R}}\to{\mathbb{R}}_{>0}$ is a smooth function and
$g_{ab}$ is Riemannian constant curvature.  In general, a warped product
(\ref{warped}) with the negative sign in front of $dt^2$ and $g_{ab}$ an
arbitrary Riemannian metric is known as a Generalised Robertson-Walker (GRW)
spacetime~\cite{MM}.  In the spirit of this article, however, we shall restrict
to the {\em spatially locally symmetric\/} case, i.e.~where the Riemannian
metric $g_{ab}$ is locally symmetric.

In two dimensions, for example, let us consider a metric of the form 
\begin{equation}\label{warped_metric_in_2d}
\Omega^2(t)dx^2+dt^2\quad\mbox{on}\enskip{\mathbb{R}}\times{\mathbb{R}}.
\end{equation}
Splitting $\bigodot^2\!\Wedge^1$ according to type 
(\ref{symmetric_two-forms_according_to_type}), the Killing operator is
$$X\,dx+\xi\,dt\stackrel{{\mathcal{K}}}{\longmapsto}
\left[\!\begin{array}{c}X_x+\Omega^2\Upsilon\xi\\ 
\xi_x+X_t-2\Upsilon X\\ \xi_t\end{array}\!\right],$$
where $\Upsilon\equiv\Omega^{-1}\Omega_t$. The {\em Khavkine operator\/} is
then
\begin{equation}\label{khavkine_operator}
\left[\!\begin{array}{c}p\\ q\\ r\end{array}\!\right]\mapsto
\left[\!\begin{array}{c}J_t-r\\
J_{xx}-\Omega^2\Upsilon J_t
-\Omega^2\Upsilon^\prime J-q_x+p_t-2\Upsilon p
\end{array}\!\right],\end{equation}
where, supposing that 
$\Upsilon^{\prime\prime}+2\Upsilon\Upsilon^\prime\not=0$, 
$$J\equiv\frac{p_{tt}-2\Upsilon p_t-2\Upsilon^\prime p-q_{xt}+r_{xx}
-\Omega^2(\Upsilon r_t+2\Upsilon^\prime r+2\Upsilon^2r)}
{\Omega^2(\Upsilon^{\prime\prime}+2\Upsilon\Upsilon^\prime)}.$$
\begin{prop}\label{RxR}The sequence 
$$\textstyle\Wedge^1\xrightarrow{\,{\mathcal{K}}\,}\bigodot^2\!\Wedge^1
\xrightarrow{\,{\mathrm{Khavkine}}\,}
\begin{array}{c}\Wedge^0\\[-4pt] \oplus\\[-2pt] \Wedge^0\end{array}$$
is a locally exact complex on ${\mathbb{R}}\times{\mathbb{R}}$. 
\end{prop}
\begin{proof} A computation shows that the composition
$$\textstyle\Wedge^1\xrightarrow{\,{\mathcal{K}}\,}\bigodot^2\!\Wedge^1
\xrightarrow{\,J\,}\Wedge^0$$
sends $X\,dx+\xi\,dt$ to $\xi$. Therefore, the first component $J_t-r$ of the
Khavkine operator forces $r=\xi_t$. This effectively eliminates $\xi$ from
$X\,dx+\xi\,dt$ and the second component of the Khavkine operator is exactly
what is needed as a consequence of
$$X\,dx\longmapsto\left[\!\begin{array}{c}X_x\\
X_t-2\Upsilon X\end{array}\!\right]$$
defining a flat connection.
\end{proof}
Notice that the integrability operator (\ref{khavkine_operator}) is of fourth
order. The key to Khavkine's construction is the following observation.
\begin{lemma}\label{khavkine_key}
For a general semi-Riemannian metric, the composition
$$\Wedge^1=\begin{picture}(6,6)(0,0)
\put(0,0){\line(1,0){6}}
\put(0,6){\line(1,0){6}}
\put(0,0){\line(0,1){6}}
\put(6,0){\line(0,1){6}}
\end{picture}\xrightarrow{\,{\mathcal{K}}\,}
\begin{picture}(12,6)(0,0)
\put(0,0){\line(1,0){12}}
\put(0,6){\line(1,0){12}}
\put(0,0){\line(0,1){6}}
\put(6,0){\line(0,1){6}}
\put(12,0){\line(0,1){6}}
\end{picture}\xrightarrow{\,{\mathcal{C}}\,}
\begin{picture}(12,12)(0,2)
\put(0,0){\line(1,0){12}}
\put(0,6){\line(1,0){12}}
\put(0,12){\line(1,0){12}}
\put(0,0){\line(0,1){12}}
\put(6,0){\line(0,1){12}}
\put(12,0){\line(0,1){12}}
\end{picture}\ni X_{abcd}\longmapsto X_{ab}{}^{ab}\in\Wedge^0$$
is a homomorphism sending $\sigma_b$ to $-(\nabla^bR)\sigma_b$, where $R$ is
the scalar curvature.
\end{lemma}
\begin{proof} The curvature of the prolongation connection acquires an extra
term (beyond the locally symmetric formula (\ref{curvature})):
$$(\nabla_a\nabla_b-\nabla_b\nabla_a)
\left[\!\begin{array}{c}\sigma_c\\ \mu_{cd}\end{array}\!\right]
=\left[\!\begin{array}{c}0\\ 
2R_{ab}{}^e{}_{[c}\mu_{d]e}+2R_{cd}{}^e{}_{[a}\mu_{b]e}
-(\nabla^eR_{abcd})\sigma_e
\end{array}\!\right]$$
and the conclusion follows from~(\ref{straightforward_calculation}).
\end{proof}
Alternatively, one can compute the composition 
$\begin{picture}(12,6)(0,0)
\put(0,0){\line(1,0){12}}
\put(0,6){\line(1,0){12}}
\put(0,0){\line(0,1){6}}
\put(6,0){\line(0,1){6}}
\put(12,0){\line(0,1){6}}
\end{picture}\to
\begin{picture}(12,12)(0,2)
\put(0,0){\line(1,0){12}}
\put(0,6){\line(1,0){12}}
\put(0,12){\line(1,0){12}}
\put(0,0){\line(0,1){12}}
\put(6,0){\line(0,1){12}}
\put(12,0){\line(0,1){12}}
\end{picture}\to\Wedge^0$ to be 
\begin{equation}\label{the_composition}
h_{ab}\longmapsto
2\big[\Delta h_a{}^a-\nabla^a\nabla^bh_{ab}+R^{ab}h_{ab}\big]\end{equation}
and check that, when $h_{ab}=\nabla_{(a}\sigma_{b)}$, this formula
gives~$-(\nabla^bR)\sigma_b$. For a warped product (\ref{warped}) with the
metric on $M$ being locally symmetric, it is clear that the scalar curvature
$R$ is a function of $t$ alone. Therefore, supposing that $\partial R/\partial
t\not=0$, we may define
$$J(h_{ab})\equiv
-2\frac{\Delta h_a{}^a-\nabla^a\nabla^bh_{ab}+R^{ab}h_{ab}}
{\partial R/\partial t}$$
and, in case that $h_{ab}=\nabla_{(a}\sigma_{b)}$, conclude from 
Lemma~\ref{khavkine_key} that 
$$J(h_{ab})=\frac{\partial}{\partial t}\intprod \sigma_b,$$
which was the first step in our proof of Proposition~\ref{RxR}. More 
generally, it
is straightforward to combine this conclusion with the prolongation 
connection and Calabi operator
on $M$ to manufacture a complex
$$\textstyle \Wedge^1\xrightarrow{\,{\mathcal{K}}\,}\bigodot^2\!\Wedge^1
\begin{picture}(55,0)
\put(3,3){\vector(1,0){50}}
\put(3,6){\vector(3,1){50}}
\put(3,0){\vector(3,-1){50}}
\end{picture}\raisebox{-3pt}{$\begin{array}{c}
\Wedge^0\\[-4pt]
\oplus\\[-2pt]
\begin{picture}(12,6)(0,0)
\put(0,0){\line(1,0){12}}
\put(0,6){\line(1,0){12}}
\put(0,0){\line(0,1){6}}
\put(6,0){\line(0,1){6}}
\put(12,0){\line(0,1){6}}
\end{picture}\,\Wedge^{1,0}\\[-4pt] \oplus\\
\overline{\begin{picture}(12,11)(0,2)
\put(0,0){\line(1,0){12}}
\put(0,6){\line(1,0){12}}
\put(0,12){\line(1,0){12}}
\put(0,0){\line(0,1){12}}
\put(6,0){\line(0,1){12}}
\put(12,0){\line(0,1){12}}
\end{picture}\,\Wedge^{1,0}}\,,
\end{array}$}$$
which is locally exact unless $M$ has at least one flat factor and at 
least one Hermitian factor (as in Theorem~\ref{main_thm}). The details, in the 
constant curvature case, are in~\cite{K}. 

Also treated in~\cite{K} is the $4$-dimensional Schwarzschild metric and higher
dimensional versions thereof, warped products $S^m\times N$ for which the
warping factor $\Omega^2$ multiplying the round sphere metric is annihilated by
a Killing field on the $2$-dimensional Lorentzian manifold~$N$. Finally,
although not a warped product, the Kerr metric is treated in~\cite{AABKW}.

\subsection{The deformation operator}\label{deformation}
Apart from (\ref{calabi}), there is another natural differential operator
$\,\begin{picture}(12,6)(0,0)
\put(0,0){\line(1,0){12}}
\put(0,6){\line(1,0){12}}
\put(0,0){\line(0,1){6}}
\put(6,0){\line(0,1){6}}
\put(12,0){\line(0,1){6}}
\end{picture}\to
\begin{picture}(12,12)(0,2)
\put(0,0){\line(1,0){12}}
\put(0,6){\line(1,0){12}}
\put(0,12){\line(1,0){12}}
\put(0,0){\line(0,1){12}}
\put(6,0){\line(0,1){12}}
\put(12,0){\line(0,1){12}}
\end{picture}\,$, obtained by deforming the metric and observing the change 
in the Riemann curvature tensor. Specifically, if 
$\widetilde{g}_{ab}=g_{ab}+\epsilon h_{ab}$, then we find that
\begin{equation}\label{R_deformed}\widetilde{R}_{abcd}
=R_{abcd}-\frac{\epsilon}2\left[\!\begin{array}{c}
\nabla_{(a}\nabla_{c)}h_{bd}-\nabla_{(b}\nabla_{c)}h_{ad}
-\nabla_{(a}\nabla_{d)}h_{bc}+\nabla_{(b}\nabla_{d)}h_{ac}\\[5pt]
{}+R_{ab}{}^e{}_{[c}h_{d]e}+R_{cd}{}^e{}_{[a}h_{b]e}\end{array}\!\right]
+{\mathrm{O}}(\epsilon^2)\end{equation}
and, initially, one might be disturbed by the resulting operator
$h_{ab}\mapsto[\quad]_{abcd}$ having the opposite sign from (\ref{calabi}) in
front of the curvature terms.  To reconcile this discrepancy, let us consider
the deformation of scalar curvature
$$\textstyle
\widetilde{R}=\widetilde{g}^{ac}\widetilde{g}^{bd}\widetilde{R}_{abcd}
=(g^{ac}-\epsilon h^{ac})(g^{bd}-\epsilon h^{bd})
(R_{abcd}-\frac\epsilon2[\quad]_{abcd})+{\mathrm{O}}(\epsilon^2),$$
which simplifies as 
\begin{equation}\label{three}
\widetilde{R}=R-\epsilon[\Delta h_a{}^a-\nabla^a\nabla^bh_{ab}+R^{ab}h_{ab}]
+{\mathrm{O}}(\epsilon^2)\end{equation}
and we have obtained the operator (\ref{the_composition}), exactly as 
expected. The Killing operator is directly related to the Lie derivative of the 
metric: for any vector field $X^a$, 
\begin{equation}\label{two}{\mathcal{L}}_Xg_{ab}=2\nabla_{(a}X_{b)}.
\end{equation}
That the association of scalar curvature to a metric is co\"ordinate-free is 
also expressible in terms of Lie derivatives,
\begin{equation}\label{one}
\widetilde{g}_{ab}=g_{ab}+\epsilon{\mathcal{L}}_Xg_{ab}\enskip
\Longrightarrow\enskip\widetilde{R}=R+\epsilon{\mathcal{L}}_XR
+{\mathrm{O}}(\epsilon^2)=R+\epsilon X^e\nabla_eR+{\mathrm{O}}(\epsilon^2).
\end{equation}
Comparing (\ref{one}), (\ref{two}), and (\ref{three}) gives a geometric proof
of Lemma~\ref{khavkine_key}. Indeed, this is the interpretation of the operator
(\ref{the_composition}) given by Khavkine~\cite{K}.

An alternative potential reconciliation, already suggested by Gasqui and
Goldschmidt \cite[p.~207]{GG} (see also \cite[\S2.2]{Khavkine}), is to 
consider the operator
$$\textstyle D_g:\bigodot^2\!\Wedge^1\to\Wedge^2\otimes\Wedge^2$$
obtained by deforming the Riemann curvature viewed as {\em an operator\/}
$R_{ab}{}^{cd}:\Wedge^2\to\Wedge^2$, as in~(\ref{R}), and then lowering the
last two indices with the {\em undeformed\/} metric~$g_{ab}$.  {From}
(\ref{R_deformed}) we obtain
$$h_{ab}\stackrel{D_g}{\longmapsto}-\frac12\left[\!\begin{array}{c}
\nabla_{(a}\nabla_{c)}h_{bd}-\nabla_{(b}\nabla_{c)}h_{ad}
-\nabla_{(a}\nabla_{d)}h_{bc}+\nabla_{(b}\nabla_{d)}h_{ac}\\[5pt]
{}-3R_{ab}{}^e{}_{[c}h_{d]e}+R_{cd}{}^e{}_{[a}h_{b]e}\end{array}\!\right].$$
Unfortunately, this does not generally satisfy the interchange symmetry to end
up in
$\,\begin{picture}(12,12)(0,2)
\put(0,0){\line(1,0){12}}
\put(0,6){\line(1,0){12}}
\put(0,12){\line(1,0){12}}
\put(0,0){\line(0,1){12}}
\put(6,0){\line(0,1){12}}
\put(12,0){\line(0,1){12}}
\end{picture}\,$. 
Indeed, elementary representation theory shows that this happens if and only if
$g_{ab}$ is constant curvature~(\ref{cc}).  Whilst symmetrising over
$ab\leftrightarrow cd$ forces the result into 
$\,\begin{picture}(12,12)(0,2)
\put(0,0){\line(1,0){12}}
\put(0,6){\line(1,0){12}}
\put(0,12){\line(1,0){12}}
\put(0,0){\line(0,1){12}}
\put(6,0){\line(0,1){12}}
\put(12,0){\line(0,1){12}}
\end{picture}\,$
$$\textstyle\bigodot^2\!\Wedge^1\xrightarrow{\,D_g\,}
\Wedge^2\otimes\Wedge^2\longrightarrow
\,\begin{picture}(12,12)(0,2)
\put(0,0){\line(1,0){12}}
\put(0,6){\line(1,0){12}}
\put(0,12){\line(1,0){12}}
\put(0,0){\line(0,1){12}}
\put(6,0){\line(0,1){12}}
\put(12,0){\line(0,1){12}}
\end{picture}$$
and gives ($-\frac12\times$) the Calabi operator, this observation does not
appear to lead anywhere.  About as much as one can say, as already observed by
Gasqui-Goldschmidt and others, is that {\em in the constant curvature case\/},
it is clear that the composition
$$\textstyle\Wedge^1\xrightarrow{\,{\mathcal{K}}\,}
\bigodot^2\!\Wedge^1\xrightarrow{\,D_g\,}
\Wedge^2\otimes\Wedge^2$$
vanishes since
${\mathcal{L}}_X(\delta_a{}^c\delta_b{}^d-\delta_b{}^c\delta_a{}^d)=0$ for any
vector field~$X$.

\subsection{Third order compatibility}\label{third_order}
The differential operator ${\mathcal{L}}$ in the complex (\ref{key_complex}) is
second order and, as detailed in Theorem~\ref{main_thm}, is often sufficient
locally to characterise the range of the Killing operator.  That is not to say
that a higher order differential operator might not more easily provide such
integrability conditions.  Indeed, the deformation operator of
\S\ref{deformation} suggests that one might augment the Calabi operator with
$$\begin{picture}(12,6)(0,0)
\put(0,0){\line(1,0){12}}
\put(0,6){\line(1,0){12}}
\put(0,0){\line(0,1){6}}
\put(6,0){\line(0,1){6}}
\put(12,0){\line(0,1){6}}
\end{picture}\begin{picture}(30,0)
\put(5,3){\vector(1,0){25}}
\put(5,0){\vector(3,-2){25}}
\put(18,10){\makebox(0,0){${\mathcal{C}}$}}
\put(18,-20){\makebox(0,0){$\widetilde{\mathcal{C}}$}}
\end{picture}
\begin{array}[t]{c}\begin{picture}(12,12)(0,2)
\put(0,0){\line(1,0){12}}
\put(0,6){\line(1,0){12}}
\put(0,12){\line(1,0){12}}
\put(0,0){\line(0,1){12}}
\put(6,0){\line(0,1){12}}
\put(12,0){\line(0,1){12}}
\end{picture}\\[-2pt] \oplus\\[-3pt]
\begin{picture}(12,12)(0,2)
\put(0,0){\line(1,0){12}}
\put(0,6){\line(1,0){18}}
\put(0,12){\line(1,0){18}}
\put(0,0){\line(0,1){12}}
\put(6,0){\line(0,1){12}}
\put(12,0){\line(0,1){12}}
\put(18,6){\line(0,1){6}}
\end{picture}\end{array}\enskip\raisebox{-12pt}{,}$$
where $\widetilde{\mathcal{C}}$ deforms $\nabla_aR_{bcde}$ 
(with symmetries in $\begin{picture}(18,12)(0,2)
\put(0,0){\line(1,0){12}}
\put(0,6){\line(1,0){18}}
\put(0,12){\line(1,0){18}}
\put(0,0){\line(0,1){12}}
\put(6,0){\line(0,1){12}}
\put(12,0){\line(0,1){12}}
\put(18,6){\line(0,1){6}}
\end{picture}$ by the Bianchi identity). This option is pursued by Gasqui and 
Goldschmidt who demonstrate \cite[Th\'eor\`eme~7.2]{GG} that the resulting
third order integrability conditions are sufficient locally to characterise 
the range of the Killing operator on all Riemannian locally symmetric spaces. 
It is straightforward to obtain this result by augmenting the curvature 
(\ref{curvature}) of the prolongation connection. Details may be found in our 
follow-up article~\cite{CELM2}. 

\subsection{Projective differential geometry}\label{projective}
Although hidden in our presentation so far, several aspects of our
constructions are projectively invariant.  In particular, the Killing and
Calabi operators are the first two operators in a 
{\em Bernstein-Gelfand-Gelfand\/} sequence \cite{CD,CSS,HSSS} from projective
differential geometry.  From this point of view, the prolongation connection
(\ref{prolongation_connection}) should be rewritten as
$$\left[\!\begin{array}{c}\sigma_c\\ \mu_{cd}\end{array}\!\right]
\stackrel{\nabla_b}{\longmapsto}
\left[\begin{array}{c}\nabla_b\sigma_c-\mu_{bc}\\ 
\nabla_b\mu_{cd}+2\Rho_{b[c}\sigma_{d]}
\end{array}\right]
-\left[\begin{array}{c}0\\ W_{cd}{}^e{}_b\sigma_e\end{array}\right]\,,$$
where $\Rho_{ab}\equiv\frac1{n-1}R_{ab}$ is the projective Schouten tensor and
$W_{ab}{}^c{}_d$ is the projective Weyl tensor.  Correspondingly, the Calabi
operator (\ref{calabi}) should be rewritten as
$$h_{ab}\mapsto\!\!
\begin{array}[t]{l}
\big(\nabla_{(a}\nabla_{c)}\!+\!\Rho_{ac}\big)h_{bd}
-\big(\nabla_{(b}\nabla_{c)}\!+\!\Rho_{bc}\big)h_{ad}
-\big(\nabla_{(a}\nabla_{d)}\!+\!\Rho_{ad}\big)h_{bc}
+\big(\nabla_{(b}\nabla_{d)}\!+\!\Rho_{bd}\big)h_{ac}\\[3pt]
\enskip{}-W_{ab}{}^e{}_{[c}h_{d]e}-W_{cd}{}^e{}_{[a}h_{b]e},\end{array}$$
each line of which is projectively invariant (acting on symmetric tensors
$h_{ab}$ of weight~$2$).  The first line is the second `raw' BGG operator from
\cite{CD,CSS} and the whole Calabi operator is obtained by `normalisation'
according to~\cite{HSSS}.

Finally, we suspect that Beltrami's
theorem~\cite{B} should generalise to yield a projectively invariant
characterisation of locally symmetric metrics.

\raggedright\raggedbottom

\end{document}